\documentclass[7pt,twoside]{article}
\usepackage[dvips]{graphicx}
\usepackage{fancyhdr}
\usepackage{color}
\usepackage{amsmath}
\usepackage{amssymb}
\usepackage{tikz}
\usepackage{bm}
\usepackage{url}
\usepackage{latexsym}
\usepackage{arydshln}

\usepackage[utf8]{inputenc}

\usepackage{xcolor}
 \numberwithin{equation}{section}

\def\xxl{\color{red}}
\def\mitra{\color{olive}}
\definecolor{darkviolet}{rgb}{0.58,0,0.83} 
\def\xxlnew{\color{darkviolet}}

\newcommand{\clsp}[1]{
{ \overline{span}_{i\in I}\left\{ #1 \right\} }
}

\newcommand{\Hil}[0]{
\mathcal{H}
}
\newcommand{\norm}[2]{
\left\| #2 \right\|_{#1}
}

\newcommand{\BL}[1]{
{\mathcal B} \left( #1 \right)
}

\newtheorem{theorem}{Theorem}[section]
\newtheorem{definition}{Definition}[section]
\newtheorem{proposition}[theorem]{Proposition}
\newtheorem{lemma}[theorem]{Lemma}
\newtheorem{corollary}[theorem]{Corollary}

\newtheorem{ex}[theorem]{Example}
\newtheorem{conj.}[theorem]{Conjecture}
\newtheorem{Bsp.}{Example}[section]

\newenvironment{proof}{\noindent \bf Proof: \rm}{$ \hspace{\stretch{1}} \Box $

\vspace{1mm}}

\def\wisi{\widetilde{\psi}}

\def\wiSi{\widetilde{\Psi}}

\newcommand{\range}[1]{\mathsf{ran}\left( #1 \right)} 
\newcommand{\domain}[1]{\mathsf{dom}\left( #1 \right)} 
\newcommand{\kernel}[1]{\mathsf{ker}\left( #1 \right)} 

\begin{document}
%
%
\title{\bf\vspace{-5
pt} 
 An unbounded operator theory approach to lower frame and Riesz-Fischer sequences
}
%
%
\author{Peter Balazs, Mitra Shamsabadi}

%
%
\date{}

%
%
\maketitle
\begin{abstract}
Frames and orthonormal bases are naturally linked to bounded operators. 
To tackle unbounded operators those sequences might not be well suited. 
This has already been noted by von Neumann in the 1920ies. 
But modern frame theory also investigates other sequences, including those that are not naturally linked to bounded operators. The focus of this manuscript will be two such kind of sequences: lower frame and Riesz-Fischer sequences. We will discuss the inter-relation of those sequences.
We will fill a hole existing in the literature regarding the classification of those sequences by their synthesis operator.
We will use the idea of generalized frame operator and Gram matrix and extend it. 
We will use that to show properties for canonical duals for lower frame sequences, like e.g. a minimality condition regarding its coefficients. We will also show that other results that are known for frames can be generalized to lower frame sequences. 

To be able to tackle these tasks, we had to revisit the concept of invertibility (in particular for non-closed operators).  In addition, we are able to define a particular adjoint, which is uniquely defined for any operator.

%
\end{abstract}

%

\section{Introduction}

Frames are sequences in Hilbert spaces that allow redundant, i.e. non-unique  representations \cite{Casaz1,ole1n,duffschaef1}. 
 Over the last years, they were established as strong tool for working with bounded operators \cite{xxlframoper1,xxlcharshaare18,mitraG}. Associated to those sequences, are canonical {\em frame-related operators}, like the analysis, synthesis, frame or Gram operators. 
As the frame condition involves an upper and lower bound conditions, those operators are bounded, and stably invertible. 
If the upper frame condition, also called Bessel condition is not assumed to hold, those operators cannot be considered bounded anymore. 
 A major breakthrough in this
field was reached by Casazza et.al in 2002 \cite{casoleli1} by showing that the
 properties of lower frame sequences are related to those of the associated analysis operators, which are in this case unbounded. 
 The study \cite{xxlstoeant11} extends this approach and defines those operators for any sequence as potentially unbounded operators: In a survey approach, existing results were collected and extended with new ones for the classification of sequences by the behaviour of their frame-related operators, \cite[Section 4.1.]{xxlstoeant11}.

Still for this classification, some questions remain open
, they need a deeper insertion into the theory of unbounded operators and will be covered in this manuscript. For example,  even if $\Psi$ is an arbitrary sequence without any further assumption, the analysis operator $C_{\Psi}$ is always closed and this leads to nice classification results. 
This is not the case for the synthesis operator. Therefore, as $D_{\Psi}$ is not closed, 
 \cite[Proposition 4.2. (f)]{xxlstoeant11} had to employ properties of different, albeit related operators. 
 This observation raises the question on how to
{\em classify Riesz-Fischer sequences and lower frame sequences by properties of the synthesis operators alone, in particular not including their inverses and adjoints.}

To tackle this problem, one has to deal with not necessarily closed operators, and also discuss different notions of invertibility (which coincide for closed operators). 

Defining the adjoint (and inversion of) operators, is a significant part of unbounded operator theory. 
Traditionally, this was done for densely defined operators and closable operators, respectively. 
Here we extend the basic definitions of the adjoint and the inverse of  operators of an arbitrary unbounded operator and create some links between them and some kinds of the invertibility concepts.
Those new concepts will be used to  classify all lower frame sequences and Riesz-Fischer sequences without any further assumption by using the synthesis operators. 
Moreover, this provides us with a link to another question, namely how to prove the converse of \cite[Proposition 3.2]{casoleli1}.

This also allows us also to extend the theory of generalized frame-related operators. Those were introduced by R. Corso \cite{Corso2019,Corso2019GeneralizedFO} as a advancement of the above-mentioned approach \cite{xxlstoeant11}.
By this we can solve one of the most important topics about frames and unbounded  operators, still  open and underdeveloped: how to classify   
 minimal lower frame sequences respectively complete Riesz-Fischer sequences. 
 We will show that it is possible to derive a biorthogonal sequence for every minimal lower frame sequence. 
 
\subsection{Outlines}
 We start our observations by presenting a short review of definitions and basic concepts from literature in Section \ref{sec:prel0}.
 In Section 3, we will discuss an important distinction of invertibility concepts needed for this paper. 
Moreover, we connect the invertibility of a densely defined operator to its closure and adjoint operators.
In Section 4,
we present restricted frame-related operators for arbitrary sequences.
We characterize lower frame sequences and Riesz-Fischer bases using their synthesis operators  in Section 5. 
In Section 6, we review some basic definitions and results of positive and symmetric sesquilinear forms and link them to sequences.  
In passing, we will show that the coefficients using the canonical dual fulfill a minimality condition. 

\section{Preliminaries} \label{sec:prel0}
Throughout this paper,  $\Hil_1$ and $\Hil_2$ are Hilbert spaces, $I$  a countable index set and 
$\pi_{V}$ denotes the orthogonal projection from a separable Hilbert space $\mathcal{H}$
onto a closed subspace $V$.
The sequence $\{\delta_i\}_{i\in I}$ denotes the canonical basis of $\ell^2$ and $\{e_i\}_{i\in I}$ an orthonormal basis for $\Hil$. 
For a given operator $T$, 
we denote its domain, range and the null space  by $\domain{T}$, $\range{T}$ and $\kernel{T}$, respectively.


\subsection{Unbounded Operators}


Remember that a linear  operator $T:\domain{T}\subseteq \Hil_1 \rightarrow \Hil_2$ is called \emph{densely defined} if $\overline{\domain{T}}=\Hil_1.$  Moreover, an operator $S:\domain{S}\subseteq \Hil_1 \rightarrow \Hil_2$ is called an \emph{extension} of $T$ (or $T$ is a \emph{restriction} of $S$) if $\domain{T}\subseteq \domain{S}$ and $Tf=Sf$, for all $f\in \domain{T}$, in symbol 
$T\subseteq S$ or $ S\supseteq T.$ We denote the Hilbert space $\overline{\domain{T}}$ (with induced topology of $\Hil$) by $\Hil_{T}$. This means that  $T:\Hil_{T}\subseteq \Hil_1 \to \Hil_2$ is always densely defined.

\begin{definition}  \label{def:adjoint0} \cite{ka95-1}
Let $T:\domain{T}\subseteq \Hil_1 \rightarrow \Hil_2$ be a 
linear operator, then let (formally)
$$\domain{T^*} = \left\{ k \in \Hil_2 : h \mapsto \left< T h , k \right> \mbox{\small is a bounded functional for all } h \in\domain{T} \right\} .$$
We call a linear operator $S : \Hil_2 \rightarrow \Hil_1$ an adjoint of $T$ if 
\begin{equation} \label{eq:ajdoint} \left< T h , k \right> = \left< h , S k \right> \text{ for all } h \in \domain{T}, k \in \domain{S}.
\end{equation}
Clearly, $\domain{S} \subseteq \domain{T^*}$. 
If $T$ is densely defined, then $T^*$ denote its (unique) adjoint. 
\end{definition}
There can be many adjoints \cite{ka95-1} except if the operator is densely defined. 
The adjoint operator is well-defined and unique, whenever  $T$ is densely defined, by the Riesz representation theorem.  For two densely defined operators $T$ and $U$, if $U \subseteq T$, then $ T^{\ast} \subseteq U^{\ast}$.
 Also, by \cite[Problem III.5.26 ]{ka95-1}, if $S:\domain{S}\subseteq \Hil_3\to \Hil_1$ and $TS$ densely defined, then  $S^{\ast}T^{\ast}\subseteq (TS)^{\ast}$. 

Later, in Section \ref{sec:adjoint1}, we will discuss a way to define a particular adjoint for operators that are not densely defined.



\begin{definition}\cite{ka95-1}
The graph of $T$ is the set 
$$\mathbb{G}(T):=\left\{(h,Th)\in \Hil_1\oplus \Hil_2, h\in \domain{T}\right\}.$$
An operator $T$ is called \emph{closed} if it has a \emph{closed graph}. It is called \emph{closable}, if it has a closed extension. This means that the closure of its graph is again a graph. 
 We denote the closure of $T$ by $\overline{T}$. Moreover,  
\begin{eqnarray}\label{Tbastform} 
    \domain{\overline{T}}=\left\{ x\in \Hil_1:\exists \{x_n\}\subset \domain{T},~x_n\rightarrow x, 
   \{Tx_n\} \text{ Cauchy sequence} \right\}.
\end{eqnarray}
In this case,  $$\overline{T}x=\lim_{n\rightarrow \infty} Tx_n.$$ 
\end{definition}

In particular, of course \begin{eqnarray}\label{dom:=clo}\domain{T}\subseteq \domain{\overline{T}}\subseteq  \overline{\domain{T}}\end{eqnarray}
and so, if $\overline{T}$ is densely defined then $T$ is densely defined.
In particular, for closable operators, $\Hil_{T} = \Hil_{\overline{T}} $.

 Also from \eqref{Tbastform} it follows that \begin{eqnarray}\label{ran=clo}\range{T}\subseteq \range{\overline{T}}\subseteq \overline{\range{T}}.
 \end{eqnarray}
Hence, if $\overline{T}$ has  dense range if and only if $T$ has dense range.
\begin{definition}\cite[Definition 13.7]{rudinengl}
The operator $T$ is called \emph{symmetric}, 
\begin{eqnarray}\label{def:sym}
\langle Tf,g\rangle=\langle f,  Tg\rangle, \qquad (f, g\in \domain{T})
\end{eqnarray}
A densely defined symmetric operator $T$ satisfies 
$T\subseteq T^*$. 
Moreover, $T$ is called self-adjoint if, in addition, $T^*=T$. 
A symmetric operator $T$ is called \emph{essentially self-adjoint} if $T$ is densely defined and its closure $T^{**}$ is self-adjoint. Also, $T$ is called positive (resp. non-negative) 
if $\left\langle Tf,f\right\rangle > 0$ (resp. $\left\langle Tf,f\right>\geq0$)  for all $f \in \domain{T}$.  A positive operator on a complex Hilbert space is symmetric and has a self-adjoint extension which is a positive operator.
\end{definition}
 Note that for bounded operators a positive operator is self-adjoint. For unbounded operator usually, see e.g. \cite{ka95-1}, the positivity is defined including a symmetric 
assumption. As we will go beyond the densely defined setting (see Section \ref{sec:adjoint1}), we haven chosen this definition.


We also refer to the following results frequently.
\begin{proposition}\label{4propertdens}
Let $T$ be a densely defined operator. The following assertions hold.
\begin{enumerate}
   \item \cite[Proposition X.1.6]{conw1}   $T^*$ is closed.
    \item \cite[Proposition X.1.6]{conw1}  $T$ is closable if and only if $T^*$ is densely defined.
    \item   \cite[Theorem VIII.1]{rs1}  If $T$ is closable, then its closure is $T^{**}$ and $(
    \overline{T})^*=T^*$.
    \item  \cite[Proposition X.1.13]{conw1}  $(\range{T})^{\perp}=\kernel{T^*}$. Moreover, if $T$ is closable then $(\range{T^*})^{\perp}=\kernel{\overline{T}}$.

\item \cite[Theorem III.5.13]{ka95-1} If $T$ is closed, then 
$ran(T)$ is closed if and only if  $ran(T^*)$ is closed.
\item \label{TT*self}\cite[Theorem V.3.24]{ka95-1} If $T$ is closed, then  $T^*T$ is a closed, densely defined and self-adjoint operator. 
\end{enumerate}
\end{proposition}

In this paper, we also need the definition of the pseudo-inverse of unbounded operators \cite[Lemma 1.1]{fre14}.
\begin{proposition}\label{pseudogeneral}
Let $T:\domain{T}\subseteq \Hil_1\to \Hil_2$ be a densely defined, closed operator. There exists an  unique operator $T^{\dag}$ from $\Hil_2$ to $\Hil_1$ such that 
\begin{eqnarray*}
\overline{\range{T^{\dag}}}=\kernel{T}^{\perp}, \qquad \kernel{T^{\dag}}=\range{T}^{\perp},\qquad T T^{\dag} T = T.
\end{eqnarray*}
\end{proposition}
The operator $T^{\dag}$ is called the  {\em pseudo-inverse operator} of $T$. It is bounded if and only if $\range{\overline{T}}$ is closed, \cite[Corollary 1.2]{fre14}. 

We can extend it to any closable $T$. If we restrict the operator to its closed domain, it is densely defined. The above result can be applied for $\overline{T}$. So we can formulate 
 \begin{corollary}\label{pseudogeneral}
Let $T:\domain{T}\subseteq \Hil_1\to \Hil_2$ be a closable operator. There exists an operator $\overline{T}^{\dag}$ from $\Hil_2$ to $\Hil_1$ such that 
\begin{eqnarray*}
\overline{\range{\overline{T}^{\dag}}}=\kernel{\overline{T}}^{\perp}\cap \Hil_{T}, \qquad \kernel{\overline{T}^{\dag}}=\range{ T}^{\perp},\qquad T\overline{T}^{\dag}T=T.
\end{eqnarray*}
\end{corollary}

Again we call $\overline{T}^{\dag}$ the pseudo-inverse of $T$ (which is also the pseudo-inverse of $\overline{T}$).

\subsection{Frames}
A sequence $\Psi=\{\psi_i\}_{i\in I}$ in a Hilbert space $\Hil$ is called a \textit{frame} \cite{ole1n,duffschaef1}  if there are positive constants $A\leq B<\infty$ (called lower and upper frame bound, respectively) such that 
\begin{eqnarray}\label{defnfram}
A\|f\|^2\leq \sum_{i\in I}\left|\left< f,\psi_i\right>\right|^2\leq B\|f\|^2, \qquad (f\in \Hil).
\end{eqnarray}
If $\Psi$ satisfies  the right hand side of \eqref{defnfram}, then $\Psi$ is called a \textit{Bessel sequence}. If it satisfies the left hand side, it is called a \emph{lower frame sequence} \cite{casoleli1} (or lower semi-frame \cite{antbal12}).
Let us denote the \textit{synthesis operator} (of a Bessel sequence $\Psi$) by  $D_{\Psi}:\ell^2\rightarrow \Hil$, defined by $D_{\Psi}\{c_i\}_{i\in I}=\sum_{i\in I}c_i\psi_i$ and its adjoint, the \textit{analysis operator} by  $C_{\Psi}f=\left\{\left<f,\psi_i\right>\right\}_{i\in I}$. The frame operator $S_{\Psi}f=D_{\Psi}C_{\Psi}f=\sum_{i\in I}\left< f, \psi_i\right>\psi_i$ is positive, self-adjoint and invertible, whenever $\Psi$ is a frame \cite[Lemma 5.5.1]{ole1n}. Moreover, using a frame, the elements of a Hilbert space $\Hil$ can be reconstructed, i.e. there is a frame ${\Psi}^d=\{\psi_{i}^d\}_{i\in I}$ (called \textit{dual frame}) such that
$$f = \sum_{i\in I}\left< f,{\psi_i}^d\right>\psi_i,  =  \sum_{i\in I}\left< f,{\psi_i}\right>\psi^d_i, \qquad \forall f\in \Hil .$$
For every frame $\Psi$ there always exists  a dual frame  $\wiSi=S_{\Psi}^{-1}\Psi$, called the \textit{canonical dual}. 

A sequence $\Psi$ is a \textit{Riesz basis} in $\Hil$ if it is complete (i.e. for every non-zero element  $f\in \Hil$, $\sum_{i\in I}\left|\left<f,\psi_i\right>\right|^2>0$)
 and there exist positive  constants $A'\leq B'<\infty$ such that for every finite scalar sequences $c=\{c_i\}_{i\in I}$ one has
 $$ A'\|c\|^2\leq \left\| D_{\Psi}c\right\|^2\leq B'\|c\|^2.$$
 It is called a \emph{Riesz-Fischer sequence} if the left hand side is fulfilled (the right hand side is just the Bessel condition).
 Moreover,  a sequence $\Psi$ 
 \begin{enumerate}
     \item 
is \textit{minimal} if for all $j\in I$, $\psi_j\notin \overline{\textrm{span}}\{\psi_i\}_{i\neq j}$.
 \item is \textit{$\omega$-independent} if $\sum_{i\in I}c_i\psi_i=0$ implies that $c_i=0$, for all $i\in I$.
 \item has a \textit{biorthogonal sequence}, if there is a sequence $\Phi=\{\phi_i\}_{i\in I}$ such that $\left<\psi_i,\phi_j\right>=\delta_{ij}$, for all $i,j\in I$. 
  \end{enumerate}

	\subsubsection{Arbitrary sequences and related operators}
Naturally, other classes of sequences exist, where the frame condition is not satisfied. Still the frame-related operators $D_{\Psi}$, $C_{\Psi}$ and $S_{\Psi}$ can be defined, but are, in general, unbounded: 

\begin{definition}\cite{xxlstoeant11}
Let $\Psi$ be a sequence in $\Hil$. Then
\begin{enumerate}
    \item the \emph{synthesis operator} 
		$D_{\Psi}:\domain{D_{\Psi}}\subseteq \ell^2\rightarrow \Hil$ is defined by 
	 $D_{\Psi}\{c_i\}_{i\in I}=\sum_{i\in I}c_i\psi_i$, where  $\domain{D_{\Psi}}=\left\{\{c_i\}_{i\in I}\in \ell^2: \sum_{i\in I}c_i\psi_i \text{ converges in } \Hil\right\}$. 
      \item the \emph{analysis operator} $C_{\Psi}:\domain{C_{\Psi}}\subseteq \Hil\rightarrow \ell^2$ is defined by $C_{\Psi}f=\left\{\left<f,\psi_i\right>\right\}_{i\in I}$, where $\domain{C_{\Psi}}=\left\{f\in \Hil: \left\{\left<f,\psi_i\right>\right\}_{i\in I}\in \ell^2\right\}$.
      \item the \emph{frame operator} $S_{\Psi}:\domain{S_{\Psi}}\subseteq \Hil\rightarrow \Hil$ is defined by $S_{\Psi}f=\sum_{i\in I}\left< f, \psi_i\right>\psi_i$, where $\domain{S_{\Psi}}=\left\{f\in \Hil:\sum_{i\in I}\left< f, \psi_i\right>\psi_i \text{ converges in}\   \Hil \right\}$.
\end{enumerate}
\end{definition}
 The following preliminary result will be applied throughout the paper frequently:
\begin{proposition}\cite{xxlstoeant11}\label{xxlstoeant11}
Let $\Psi\subseteq \Hil$ be  an arbitrary sequence. The following assertions hold.
\begin{enumerate}
    \item $C_{\Psi}=D_{\Psi}^*$. 
    \item $D_{\Psi}$ is closed or equivalently $C_{\Psi}$ is closed.
    \item $C_{\Psi}$ is densely defined if and only if $D_{\Psi}$ is closable. In this case, $D_{\Psi} \subseteq C_{\Psi}^*$.
     \item $D_{\Psi}$ is closed if and only if $C_{\Psi}$ is densely defined and $D_{\Psi}=C_{\Psi}^*$.
    \item If $C_{\Psi}$ is densely defined, then $S_{\Psi}$ is closable.
\end{enumerate}
\end{proposition}

By using the potentially unbounded operators $C_\Psi$ and $D_\Psi$ we have the following 
definition.

\begin{definition}
Let $\Psi=\{\psi_i\}_{i\in I}\subseteq \Hil$. We say that
\begin{enumerate}
    \item $\Psi$ is a \emph{lower frame sequence}, if  there exists a constant $A>0$ such that \begin{eqnarray}\label{lowprop}
        \sqrt{A} \cdot \|f\|\leq \| C_\Psi f\| , \qquad (f\in \domain{C_{\Psi}}).
    \end{eqnarray}
     \item $\Psi$ is an  \emph{exact} sequence if for every $k\in I$ we have that $$\overline{\textrm{span}}_{i\in I}\{\psi_i\}\neq\overline{\textrm{span}}_{i\in I , i\neq k}\{\psi_i\}.$$

    \item $\Psi$ is a \emph{Riesz-Fischer  sequence}, if  there exists a constant $A'>0$ such that $\sqrt{A'} \cdot \| c \| \le \| D_\Psi c \|$ for all  ﬁnite scalar sequences $c=\{c_i\}_{i\in I}\in \domain{D_{\Psi}}$. 
\end{enumerate}
\end{definition}





Note that the sequence $\left\{\left< f , \psi_k \right>\right\}_{k\in I}$ is always defined pointwise, therefore $\| C_\Psi f \|$ is either finite or infinite - i.e. $f \in \domain{C_\Psi}$ or $\norm{}{C_\Psi f} = \infty$
- so we could give the definition of a lower frame sequence by \eqref{lowprop} for all $f\in \Hil$ and so any lower frame sequence is complete. In particular, a lower frame sequence is exact if it stops to be a lower frame sequence when an arbitrary element is removed.

The following connection of minimal and exact sequences is obvious:
\begin{corollary}\label{ex=min} Every minimal sequence is exact, and vice-versa.
\end{corollary}

The reader should keep in mind that - in general - $S_{\Psi}$ is not densely defined and closable. But if $C_{\Psi}$ is densely defined, then $S_{\Psi}$ is closable (\cite[Proposition 3.3 (vii)]{xxlstoeant11}). Moreover,  by applying the fact that $\domain{S_{\Psi}}\subseteq \domain{C_{\Psi}}$, it can be seen that if $S_{\Psi}$ is densely defined, then $C_{\Psi}$ is densely defined.
So in particular, if $S_\Psi$ is densely defined, it is closable.



\section{New Results and Concepts for Unbounded Operators}
 Here we collect new general concepts and their properties in the theory of unbounded operators, motivated but not directly part of frame theory.

 \subsection{Adjoints for Unbounded Operators with Arbitrary Domains} \label{sec:adjoint1}

As seen in Definition \ref{def:adjoint0}, the (maximal) domain of the adjoint can always be given. It is known, that for operators that are not densely defined, there are several adjoints. Here, we will pick a particular one: 


\begin{definition}
Let $T:\domain{T}\subseteq \Hil_1 \to \Hil_2$ be an unbounded operator. Then the operator ${T}^{\times} := (T|_{\Hil_T})^*$ is called the {\em canonical adjoint} of $T$.
\end{definition}
This definition is justified by the following result:

\begin{theorem}
Let $T:\domain{T}\subseteq \Hil_1 \to \Hil_2$ be an unbounded operator. The operator ${T}^{\times}$ has maximal domain - $\domain{T^*}$ - and is closed. 

For all adjoints $S$ of $T$ with maximal domain, i.e. $\domain{S} = \domain{T^*}$, we have that $S = T^\times + U$, 
where $U: \domain{T^{\times}} \rightarrow \Hil_T^\bot$ is any operator. In particular, $\norm{}{T^\times x} \le \norm{}{S x}$ for all $x \in \domain{T^{\times}}$.
\end{theorem}
\begin{proof} Clearly, 
${T_{|_{\Hil_T}}}:\domain{T}\subseteq \Hil_T \to \Hil_2$ is a densely defined operator, ${T}^{\times} := (T|_{\Hil_T})^*$ is therefore closed, and certainly an adjoint with maximal domain, i.e it fulfills \eqref{eq:ajdoint} on all of $\domain{T^*}$. 

Let $S_1$ and $S_2$ be two adjoints of $T$ with maximal domain, by \eqref{eq:ajdoint} we have that 
$$ \left< h, (S_1 - S_2) k \right> = 0, \forall h \in \domain{T}, k \in \domain{T^*}. $$

Therefore, $S = T^\times + U$ is an adjoint of $T$, as
$$ \left< h, S k \right> = \left< h, T^\times k \right> + \left< h, U k \right> = \left< h, T^\times k \right>. $$

On the other hand, let $S$ be an adjoint with maximal domain, then define $U = S-T^\times$. \end{proof}


With the same approach we can define $T^{\times \times} := \left( \left(T^{\times}\right)_{|_{\Hil_{T^{\times}}}}\right)^*$. 
Clearly, $T$ is a closable (respectively closed) operator  if and only if $T|_{\Hil_T}$ is, as $\mathbb{G}(T) = \mathbb{G}( (T|_{\Hil_T}))$  and $\Hil_T \times \Hil_2$ is a closed subspace of $\Hil_1 \times \Hil_2$.
Therefore, in this case,  $T^{\times \times}=\overline{T|_{\Hil_T}}  = \overline{T}$   on their domains- but still being considered as operators acting between different spaces.


Actually, if $T$ is densely defined operator then $T^{\times}$ is exactly $T^*$, if it is also closable then $T^{\times \times} = T^{* *} = \overline{T}$.

\begin{lemma}\label{cloT}
Assume that $T:\domain{T}\subseteq \Hil_1\to \Hil_2$ is an unbounded operator. Then $T^{\times}$ has closed range if and only if $T^{\times \times}$ has closed range.
In particular, if $T$ is a closable densely defined operator, then $\overline{T}$ has closed range if and only if $T^*$ has closed range. 
\end{lemma}
\begin{proof}
As we mention above, for every unbounded operator $T$, the operator $T^{\times}:\domain{T^{\times}}\subseteq \Hil_{T^{\times}}\to \Hil_{T}$ is closed and densely defined operator and by Proposition \ref{4propertdens} (6), $T^{\times}$ has closed range if and only if $T^{\times \times}$ has.

\end{proof}


So now, we can define a well-defined and unique adjoint, even if $T$ is not uniquely defined. So, for two operators $T$ and $U$, if $U \subseteq T$, then $ T^{\times} \subseteq U^{\times}$. We can also state - like for densely defined operators - that $U^{\times}T^{\times}\subseteq (TU)^{\times}$. We can also revisit symmetric operators: 
\begin{definition}
The operator $T$ is called \emph{symmetric}, if 
\begin{eqnarray}\label{def:sym}
\langle Tf,g\rangle=\langle f,  Tg\rangle, \qquad (f, g\in \domain{T}).
\end{eqnarray}
Moreover, $T$ is called 
\emph{relaxed self-adjoint} if, in addition, $T^{\times}=T$. 
A symmetric operator $T$ is called \emph{relaxed essentially self-adjoint} if  its closure $T^{\times\times}$ is  relaxed self-adjoint.
\end{definition}
 It is worthwhile to mention that an operator $T$ is symmetric if and only if 
$T\subseteq T^{\times}$.  Indeed,   $T$ is a symmetric operator if and only if for all $f,g\in \domain{T}$,
\begin{eqnarray*}\label{*}
\left<Tf,g\right>=\left< f,Tg\right>=\left< f, T|_{\Hil_T}g\right>
\end{eqnarray*}
if and only if  $T\subseteq \left(T_{\Hil_T}\right)^*=T^{\times}$.

\subsection{On the Invertibility of Unbounded Operators}
Taking into account that we do {\em not} consider only closed operators, let us mention that one has to be careful about the concept of invertibility. There are two different concepts in the literature. Does it mean to be invertible everywhere (so  bijective), or just on its range (thereby being injective)?
Even for bounded operators it is not always clear, which definition is used. For unbounded non-closed operators it is even more important to be precise here. 

In this paper we will look e.g. at the lower frame sequence case, which means that we have a lower bound for the analysis operator. It is therefore invertible, but  this does not imply that the operator has to be bijective. For
 our purposes it is important to distinguish the following concepts, which are not always well separated in the literature:
\begin{definition}
An operator $T:\domain {T}\subseteq \Hil_1\rightarrow \Hil_2$ is called
\begin{enumerate}
    \item 
 \emph{boundedly invertible} (we denote this by (BI))
 , if there is a bounded operator $S:\Hil_2 \rightarrow \Hil_1$ such that $TSg=g$, for all $g\in \Hil_2$ and $STf=f$, for all $f\in \domain{T}$.

\item   \emph{bounded from below} (bb) 
if there is $m>0$ such that for all $f\in \domain{T}$
$$m\|f\|\leq \|Tf\| .$$

\item   \emph{boundedly invertible  onto its range 
} 
(denote this by (BIR)), if  $T$ has closed range and $T:\domain{T} \subseteq \Hil_1 \rightarrow \range{T}$ is (BI). 
\end{enumerate}
\end{definition}
As a result if an operator is (bb) then clearly there is an inverse $T^{-1} : \range{T} \rightarrow \domain{T}$ which is bounded. 

In the literature these notions are not clearly distinguished, see e.g. \cite{gohbgol1} where 'invertible' is defined to be (BI) or e.g. \cite{ka95-1} where it means (bb).



Using this new notation let us note that

\begin{proposition}\label{sec:boundinvertunb1}
\cite[Proposition 1.15.]{conw1}  
The operator 
$T:\domain {T}\subseteq \Hil_1\rightarrow \Hil_2$
 is (BI) if and only if $T$ is closed and 
bijective. 
\end{proposition}
This, in parts, motivated the introduced concept of (BIR), in the sense that we have: 
\begin{corollary}\label{sec:boundinvertunb2}
The operator 
$T:\domain {T}\subseteq \Hil_1\rightarrow \Hil_2$
 is (BIR) if and only if $T$ is injective, closed and 
 has closed range.
\end{corollary}

We will also use 
\begin{proposition} \label{prop:injcloran1}
\cite[Proposition 2.14]{hand20}  Let $T$ be a closed operator. Then $T$ is (bb) if and only if $T$ is injective and has closed range. 
\end{proposition}

This means that for closed operators (bb) and (BIR) are equivalent.\\ 

 By  Proposition \ref{4propertdens} and Proposition \ref{prop:injcloran1} we can see that if $T$ is a closed and densely defined operator then $T$ is surjective if and only if $T^*$ is (bb). In the following we  want to  extend this result that is known for densely defined and closable operators. Using the concepts of Section \ref{sec:adjoint1} we extend \cite[Theorem IV.5.13]{ka95-1} to unbounded operators as a simple consequence:

\begin{proposition}
 \label{eq:surj}
Let $T:\domain{T}\subseteq \Hil_1\to \Hil_2$ be an unbounded operator. The following are equivalent:
\begin{enumerate}
    \item $T^{\times}$ is surjective (resp. (bb)).
    \item $T^{\times \times}$ is (bb) (resp. surjective).
    \item $T^{\times \times}$ (resp. $T^{\times}$) is injective and has closed range.
\end{enumerate}
In particular, if $T$ is a closable and densely defined operator, then the following are equivalent:
\begin{enumerate} 
\item Its adjoint $T^*$ is surjective (resp. (bb)).
\item  $\overline{T}$ is (bb) (resp. surjective).
\item $\overline{T}$ (resp. $T^*$) is injective and has closed range. 
\end{enumerate}
\end{proposition}
%
%

In summary we can state
\begin{corollary}\label{eq:BI}
\label{eq:BI}
Let $T:\domain{T}\subseteq \Hil_1\rightarrow \Hil_2$ be an operator. Then the following are equivalent:
\begin{enumerate}
    \item  $T$ is (BI).
    \item $\range{T}=\Hil_2$ and $T$ is (bb).
    \item $T$ is closed, $\range{T}$ is dense and $T$ is (bb).
    \item $T$ is closed, $\range{T}=\Hil_2$ and $\kernel{T}=\{0\}$.
\end{enumerate}
\end{corollary}


  

In the following we give an example of a bounded operator $T$ which is (BIR) and not (BI) and also an example which is (bb) and not (BIR).
 \begin{ex} \begin{enumerate}
     \item 
 Let $T$ be the standard shift, i.e. $$T(x_1,x_2,x_3,...)=(0,x_1,x_2,x_3,...)$$ on the space $\ell^2$ of one sided sequences. Clearly, $T$ is (bb) since $T$ is  isometric. However, $T$ is not (BI), as it is not surjective.  There is a bounded operator $S$, the inverse shift $S(x_1,x_2,x_3,...)=(x_2,x_3,...)$ such that $ST=1$, but there is no bounded operator $S$ such that $TS=1$.

 
 \item 
 
 Consider $T:\domain{T}\subseteq \ell^2\to \Hil$, given by
 $$T\{c_i\}_{i\in I}=\sum_{i=1}^{\infty}c_i(e_i+e_1), \qquad (\{c_i\}_{i\in I}\in \domain{T}).$$
 One can be seen that $T$ is densely defined operator (since finite sequences are dense in $\ell^2$) and is (bb) with lower bound 1. The adjoint is $$T^*:\domain{T^*}\subseteq \Hil \to \ell^2, \qquad (T^*f=\left\{\left<f, e_1+e_i\right>\right\}_{i\in I}).
$$
Actually, $\domain{T^*}=\overline{span}\{e_2, e_3,e_4,...\}$ and so it is not densely defined, hence $T$ is not closable and accordingly is not (BIR).

 \end{enumerate}
  \end{ex}
%

Let us show some simple results needed later:

\begin{lemma}\label{biclos}
Let $T:\domain{T}\subseteq \Hil_1 \rightarrow \Hil_2$ be a closable operator. The following assertions hold.
\begin{enumerate}
    \item $T$ is (bb) if and only if $\overline{T}$ is (bb).
    \item $T$ is densely defined if and only if $\overline{T}$ is densely defined.
    \item $T$ is symmetric  if and only if $\overline{T}$ is symmetric.
\end{enumerate}
\end{lemma}
\begin{proof}
$(1).$
Assume that $T$ is (bb) with bound $A$ and $x\in \domain{\overline{T}}$. So consider  a  converging sequence $\{x_n\}$ in $\domain{T}$ with $x_n\rightarrow x$, such that $Tx_n\rightarrow y$. Then $\overline{T}x=y$. So,
\begin{eqnarray*}
\left\|\overline{T}x\right\|^2
&=&\left\|\lim \limits_{n \rightarrow \infty} Tx_n\right\|^2\\
&=& \lim_{n \to \infty} \left\|Tx_n\right\|^2\\
&\geq & A^2\lim \limits_{n \rightarrow \infty} \|x_n\|^2\\
&=& A^2\|\lim \limits_{n \rightarrow \infty} x_n\|^2=A^2\|x\|^2.
\end{eqnarray*}
The converse is clear.

$(2).$ If $T$ is a densely defined operator, then so is every extension. 
The reverse is correct by \eqref{dom:=clo}.

$(3).$ 
For the first direction assume that $x,y\in \domain{\overline{T}}$, then by the definition of $\domain{\overline{T}}$, there are sequences $\{x_n\}$ and $\{y_n\}$ in $\domain{T}$ converging $x$ and $y$, respectively such that $\{Tx_n\}$ and $\{Ty_n\}$ converge to $\overline{T}x$ and $\overline{T}y$, respectively. Hence, 
\begin{eqnarray*}
\left\langle \overline{T}x,y\right\rangle &=& \left\langle \lim_{n\rightarrow\infty}Tx_n,\lim_{k\rightarrow\infty}y_k\right\rangle \\
&=& \lim_{n\rightarrow\infty}\lim_{k\rightarrow\infty}\left\langle Tx_n,y_k\right\rangle \\
&=& \lim_{n\rightarrow\infty}\lim_{k\rightarrow\infty}\left\langle x_n,Ty_k\right\rangle \\
&=& \left\langle \lim_{n\rightarrow\infty}x_n,\lim_{k\rightarrow\infty}Ty_k\right\rangle = \left\langle x,\overline{T}y\right\rangle.
\end{eqnarray*}
The reverse  immediately follows by $(2)$.

\end{proof}


The following result solves the relation between the boundedly invertilbility of the adjoint of an operator and its closure when the operator is  closable.
 \begin{proposition}\label{eq:bb(BI)(BIR)}
Let $T:\domain{T}\subseteq \Hil_1 \rightarrow \Hil_2$ be an unbounded operator. Then $T^{\times}$ is (BI) if and only if $T^{\times \times}$ is (BI). In particular, if $T$ is
a densely defined  closable  operator, then
 $T^*$ is (BI) if and only if $\overline{T}$
is (BI). 
\end{proposition}
\begin{proof} 
Assuming $T^{\times}$ to be (BI) implies that it is surjective and (bb) by Corollary \ref{eq:BI}. 
Applying  Proposition \ref{eq:surj} yields $T^{\times \times}$ is (bb) and surjective and so (BI) (by Corollary \ref{eq:BI}). The converse is the same argument in the opposite direction. 
\end{proof}




Because we deal with the related concepts in detail here, let us formulate \cite[Problem III.5.15]{ka95-1} in our new terminology: 
\begin{lemma}\label{bi}
Let $T:\domain {T}\subseteq \Hil_1\rightarrow \Hil_2$ be an  operator that is (bb) and has closed range. 
Then it is closed,  and consequently (BIR). 
\end{lemma}
\begin{proof}
Let $\{x_n\}$ be a sequence in $\domain{T}$ converging to $x\in \Hil_1$ and let $Tx_n\rightarrow y$. Since $T$ has closed range, so $y\in \range{T}$. Moreover,  $T^{-1}:\range{T}\rightarrow \domain{T}$ is bounded 
and then we have that $x_n=T^{-1}Tx_n\rightarrow T^{-1}y$. Therefore, $x=T^{-1}y\in \domain{T}$. 
Then 
\begin{eqnarray*}
Tx=TT^{-1}y=y 
\end{eqnarray*}
 shows that $T$ is closed.

\end{proof}


Restricting Corollary \ref{eq:BI}  to $\range{T}$
and Lemma \ref{bi},
we can present equivalent conditions provided that an unbounded operator is (BIR).
\begin{corollary} \label{cor:BIR}
Let $T:\domain{T}\subseteq \Hil_1\rightarrow \Hil_2$  be an operator. Then the following are equivalent:
\begin{enumerate}
    \item $T$ is (BIR).
     \item $\range{T}$ is closed and $T$ is (bb).
    \item $T$ is closed and $T$ is (bb).
    \item $T$ is closed,
    $\range{T}$ is closed and $\kernel{T}=\{0\}$. 
\end{enumerate}
\end{corollary}

 Again, let us add some generalization to published results\cite{ka95-1} reducing the invertibility condition to injectivity.   
\begin{theorem}\label{TT-1closed} 
Let $\Hil_1$, $\Hil_2$ be Hilbert spaces and  $T:\domain {T}\subseteq \Hil_1\rightarrow \Hil_2$ 
is injective. Then $T$ is closed if and only if $T^{-1}$ is closed.
\end{theorem}
\begin{proof}
We have that 
\begin{eqnarray*} \mathbb{G}(T)&:=&\left\{(h,T h)\in \Hil_1\oplus \Hil_2, h\in \domain{T}\right\}\\
& =& \left\{(T^{-1}h,h)\in \Hil_1\oplus \Hil_2, h\in \domain{T^{-1}}\right\} =: \mathbb{G}'(T^{-1}).
\end{eqnarray*}
 By using the canonical isomorphism from $\Hil_1 \oplus \Hil_2$ onto $\Hil_2 \oplus \Hil_1$ - i.e. $(x,y) \mapsto (y,x)$ - we have 
that $\mathbb{G}(T^{-1})$ is closed if and only if $\mathbb{G}'(T^{-1}) = \mathbb{G}(T)$ is.
\end{proof} 

By Propositions  \ref{sec:boundinvertunb1}, \ref{prop:injcloran1} and Lemma \ref{bi} we can summarize: 
\begin{corollary}\label{2of 3}
Let $T:\domain{T}\subseteq \Hil_1\rightarrow \Hil_2$ be an  injective
operator between Hilbert spaces. Then each two of the following assertions imply the third:
\begin{enumerate}
    \item $T$ is (bb).
    \item $T$ is closed.
    \item $T$ has closed range.
\end{enumerate}
In particular, in this case, $T$ is (BIR).
\end{corollary}
%

\section{The Restricted Frame-related Operators}\label{sec:restframop0}

One of the main goals of this paper is to
classify sequences where $C_\Psi$ is not necessarily densely defined. 
For that let us restrict the frame-related operators to the Hilbert space $\Hil_{\Psi}\subseteq \Hil$.
We 
link every sequence $\Psi$ with its related sequence $\pi_{\Hil_{\Psi}}\Psi$ as
a sequence in $\Hil_{\Psi}$. Consequently,
we can define the frame related operators associated to $\pi_{\Hil_{\Psi}}\Psi$. 
We denote the analysis 
operator
of the sequence $\pi_{\Hil_{\Psi}}\Psi$ in $\Hil_{\Psi}$
by 
$$C_\Psi^r:
\domain{C_\Psi^r}\subseteq \Hil_{\Psi}\to \ell^2  ,$$
so
$$C_\Psi^r (f) = C_{\pi_{\Hil_T} \Psi} (f) = \left( \left< f, \psi_k\right>_{\Hil_T}\right)_{k \in K}.$$
The synthesis operator is 
$$ D_\Psi^r 
: \domain{D_\Psi^r} \subseteq \ell^2 \rightarrow \Hil_{\Psi}, . $$
Therefore
$$D_\Psi^r c = D_{\pi_{\Hil_T} \Psi} c = \sum \limits_{k \in K} c_k \pi_{\Hil_T} \psi_k.$$

The operator  $C_\Psi^r$ is a densely defined, closed operator and   $D_\Psi^r$  is always closable (and densely defined).
Here, 
interestingly, \begin{eqnarray}\label{asso-seq}
\overline{D_{\Psi}^r}= \left(  C_\Psi^r   \right )^*={C}_{\Psi}^{\times}.
\end{eqnarray}
As a direct consequence $\pi_{\Hil_T} D_\Psi \subseteq D_\Psi^r$.

We call those operators the \emph{restricted analysis (synthesis) operator} of $\Psi$.
\\

One could consider $\pi_{\Hil_{\Psi}}\Psi$
as a sequence in $\Hil$, instead of one in $\Hil_T$. The example below illustrate that the $\Hil$-case is more complicated than $\Hil_{\Psi}$-case.  

\begin{ex}
Define the sequence
$$\Psi=\{e_1,e_1,e_2,e_1,e_2,e_3,e_1,e_2,e_3,e_4,...\}.$$
We note that $\domain{C_{\Psi}}=\{0\}=\Hil_{\Psi}$. Considering, $\pi_{\Hil_{\Psi}}\Psi=\{0\}$ as a sequence in $\Hil$, then 
$\domain{C_{\pi_{\Hil_{\Psi}}\Psi}}=\Hil=\Hil_{\Psi}^{\perp}.$
\end{ex}
Considering $\pi_{\Hil_{\Psi}}\Psi$ as a sequence in $\Hil$, we observe that
 \begin{eqnarray}\label{dom=C}\domain{C_{\pi_{\Hil_{\Psi}}\Psi}}=\left\{f\in \Hil: ~ ~  \{\left<f, \pi_{\Hil_{\Psi}}\psi_i\right>\}_{i\in I}\in \ell^2\right\}=\domain{C_{\Psi}}\oplus \Hil_{\Psi}^{\perp}.\end{eqnarray}
This illustrates that $C_{\pi_{\Hil_{\Psi}}\Psi}$ is densely defined in $\Hil$.
In addition,   $C_{\Psi}\subseteq C_{\pi_{\Hil_{\Psi}}\Psi}$ and $\pi_{\Hil_{\Psi}}D_{\Psi}\subseteq D_{\pi_{\Hil_{\Psi}}\Psi}$. Also, \begin{eqnarray}\label{Cpisi:Csi}
  C_{\pi_{\Hil_{\Psi}}\Psi}=C_{\Psi}\pi_{\Hil_{\Psi}}=C_{\Psi}^r\pi_{\Hil_{\Psi}}.
\end{eqnarray}
%
Very canonical, we define $S_\Psi^r: \Hil_\Psi \rightarrow \Hil_\Psi$ and $S_{\pi_{\Hil_{\Psi}}\Psi}:\Hil\to \Hil$. 
Since $C_{\Psi}^r$ and $C_{\pi_{\Hil_{\Psi}}\Psi}$ are densely defined in $\Hil_{\Psi}$ and $\Hil$, respectively, and as a result of Proposition  \ref{xxlstoeant11} (5),
$S_{\Psi}^r$ and $S_{\pi_{\Hil_{\Psi}}\Psi}$ are closable operators.

On the other hand $\pi_{\Hil_{\Psi}}D_{\Psi}\subseteq D_{\pi_{\Hil_{\Psi}}\Psi}$ and so
$$\pi_{\Hil_{\Psi}}S_{\Psi}\subseteq \pi_{\Hil_{\Psi}}S_{\Psi}\pi_{\Hil_{\Psi}}=\pi_{\Hil_{\Psi}}S_{\Psi}^r\pi_{\Hil_{\Psi ,}}\subseteq S_{\pi_{\Hil_{\Psi}}\Psi}.$$
Consequently, $\pi_{\Hil_{\Psi}}S_{\Psi}$ is always closable.
At the same time one can see that  $\domain{S_\Psi} \oplus \Hil_\Psi^\bot\subseteq \domain{S_{\pi_{\Hil_\Psi} \Psi}}$.

In summary, we observe the following relations:
\begin{proposition}\label{com:=op}
Let $\Psi$ be a sequence in $\Hil$. The following assertions hold.
\begin{enumerate}
\item $C_{\pi_{\Hil_{\Psi}}\Psi}$ is densely defined.
\item $\pi_{\Hil_{\Psi}}D_{\Psi}$ and $D_{\pi_{\Hil_{\Psi}}\Psi}$ are closable. 
\item $S_{\Psi}^r$ and $S_{\pi_{\Hil_{\Psi}}\Psi}$ are closable.
    \item $ C_{\Psi}^r\subseteq C_{\Psi}\subseteq C_{\pi_{\Hil_{\Psi}}\Psi}$. 
    \item $\pi_{\Hil_{\Psi}}D_{\Psi}\subseteq \pi_{\Hil_{\Psi}}D_{\pi_{\Hil_{\Psi}}\Psi}=D_{\Psi}^r \subseteq D_{\pi_{\Hil_{\Psi}}\Psi}\subseteq \overline{D_{\Psi}^r}$.    
    \item $\pi_{\Hil_{\Psi}}S_{\Psi}\subseteq S_{\Psi}^r \subseteq \pi_{\Hil_{\Psi}}S_{\Psi}\pi_{\Hil_{\Psi}}
     \subseteq  S_{\pi_{\Hil_{\Psi}}\Psi}.$
\end{enumerate}
\end{proposition}
\begin{proof}
We have only to show 
$(5)$: We just have to prove $D_{\pi_{\Hil_{\Psi}}\Psi}\subseteq \overline{D_{\Psi}^r}$ and the rest is clear.
Assume that $c\in \domain{C_{\pi_{\Hil_{\Psi}}\Psi}^*}$. Then for all $f\in \domain{C_{\pi_{\Hil_{\Psi}}\Psi}}$, the function $f\mapsto \left<C_{\pi_{\Hil_{\Psi}}\Psi}f,c\right>$ is bounded. Using item $(3)$, for all $f\in \domain{C_{\Psi}^r}$ the function $f \mapsto \left< c, C_{\Psi}^r f \right>$ is bounded and so $c\in \domain{\left(C_{\Psi}^r\right)^*}$. Then $\overline{D}_{\pi_{\Hil_{\Psi}}\Psi}\subseteq \overline{D_{\Psi}^r}$ and so $D_{\pi_{\Hil_{\Psi}}\Psi}\subseteq \overline{D_{\Psi}^r}$
\end{proof}
\begin{corollary}
Let $\Psi\subseteq \Hil$ be a sequence such that $S_{\Psi}$ be densely defined. Then $S_{\Psi}$ is closable, and $S_{\Psi}^r$, $S_{\pi_{\Hil_{\Psi}}\Psi}$ are densely defined.
\end{corollary}
\begin{proof}
Assume that $S_{\Psi}$ is a densely defined operator. Then $C_{\Psi}$ is densely defined since $\domain{S_{\Psi}}\subseteq \domain{C_{\Psi}}$. Therefore, $S_{\Psi}$ is closable by \cite[Proposition 3.3. (vii)]{xxlstoeant11}. Moreover, applying the equation $\domain{\pi_{\Hil_{\Psi}}S_{\Psi}}=\domain{S_{\Psi}}$ and Proposition \ref{com:=op} $(5)$, we have densely defined property of $S_{\Psi}^r$ and $S_{\pi_{\Hil_{\Psi}}\Psi}$.
\end{proof}

\section{Classification by the Synthesis Operator}

In this section we will fill the mentioned gap in the result \cite[Proposition 4.2. (f)]{xxlstoeant11}  regarding the classification by the synthesis operator (alone). 

If
 $T:\domain{T}\subseteq \Hil_1\rightarrow \Hil_2$ is  a densely defined and closable operator, applying Proposition \ref{eq:surj} 
 $T^*$ is (bb) if and only if 
    $\overline{T}$ is surjective. 
       For any sequence $\Psi$ in $\Hil$,  $\overline{D}_{\Psi}$ is always densely defined. 
       Assume $D_{\Psi}$ to be a closable operator then $C_{\Psi}$ is (bb) if and only if $\overline{D}_{\Psi}$ is surjective
       or $C_{\Psi}$ is surjective if and only if $\overline{D}_{\Psi}$ is (bb). Therefore:
\begin{corollary}\label{eq:clD to psi}
Let $\Psi$ be a sequence and $D_{\Psi}$ a closable operator. Then the following statements hold.
\begin{enumerate}
\item $\Psi$ is a lower frame sequence if and only if $\overline{D}_{\Psi}$ is surjective. 
\item $\Psi$ is a Riesz-Fischer sequence if and only if $\overline{D}_{\Psi}$ is injective and has closed range 
\end{enumerate}
\end{corollary}
 As a result of Corollary \ref{eq:clD to psi}, if $\Psi$ is a lower frame sequence, then $D_{\Psi}$ is surjective if and only if $\range{D_{\Psi}}$ is closed.
By Corollaries \ref{cor:BIR} and \ref{eq:clD to psi}, we can state that a sequence $\Psi$ with closable synthesis operator $D_{\Psi}$ is a Riesz-Fischer sequence if and only if $\overline{D}_{\Psi}$ is (BIR). 
Also, $\Psi$ is a lower frame sequence if and only if $C_{\Psi}$ is (BIR).

Corollary \ref{eq:clD to psi} can already be considered a much more canonical result that the classification in \cite{xxlstoeant11} - assuming $D_\Psi$ to be closable, though. It extends \cite[Proposition 4.5]{Corso2019GeneralizedFO}.  
By using the restricted analysis and synthesis operator, see Section \ref{sec:restframop0},  Corollary \ref{eq:clD to psi} it can be rephrased for arbitrary sequences, without this assumption.

\begin{corollary}\label{mathDlowriesz}
Assume that $\Psi$ is a sequence in $\Hil$. The following assertions hold.
\begin{enumerate}
\item[(1)] The following statements are equivalent:
\begin{enumerate}
    \item $\Psi$ is a lower frame sequence in $\Hil$.  \item
 $\overline{D_{\Psi}^r}$ is surjective.
 \item $\pi_{\Hil_{\Psi}\Psi}$  is a lower frame sequence in $\Hil_{\Psi}$.
 \item $\pi_{\Hil_{\Psi}\Psi}$  is a lower frame sequence in $\Hil$.

    \end{enumerate}
    \item[(2)] The following statements are equivalent:
\begin{enumerate}
    \item $\Psi$ is a Riesz-Fischer sequence in $\Hil$. 
    \item
 $\overline{D_{\Psi}^r}$ is injective and has closed range.
 \item $\pi_{\Hil_{\Psi}\Psi}$  is a Riesz-Fischer sequence in $\Hil_{\Psi}$.
 \item $\pi_{\Hil_{\Psi}\Psi}$  is a Riesz-Fischer sequence in $\Hil$.
    \end{enumerate}
\end{enumerate}
\end{corollary}
\begin{proof}
For the first part we observe

$(a\Leftrightarrow b)$ For every sequence $\Psi$, the operator $\overline{D_{\Psi}^r}$ is a closed and densely defined operator. Then $\Psi$ is a lower frame sequence in $\Hil$ if and only if $C_{\Psi}= C_{\Psi}^r
 =\left(\overline{D_{\Psi}^r}\right)^*$
is (bb) if and only if $\overline{D_{\Psi}^r}$ is onto by Proposition \ref{eq:surj}. 

$(b\Leftrightarrow c)$ 
is Corollary \ref{eq:clD to psi}.


$(b\Rightarrow d)$ Using \eqref{Cpisi:Csi}, $\pi_{\Hil_{\Psi}}\overline{D_{\Psi}^r}\subseteq \overline{D}_{\pi_{\Hil_{\Psi}}\Psi}$. The assumption shows that $\overline{D}_{\pi_{\Hil_{\Psi}}\Psi}$ is onto and so $\pi_{\Hil_{\Psi}}\Psi$ is a lower frame sequence in $\Hil$ by Corollary \ref{eq:clD to psi}.

$(d\Rightarrow a)$ 
Employing \eqref{Cpisi:Csi}, for all $f\in \domain{C_{\Psi}}$, $C_{\Psi}f=C_{\pi_{\Hil_{\Psi}}\Psi}f$. Therefore, if $\pi_{\Hil_{\Psi}}\Psi$ is a lower frame sequence in $\Hil$ or equivalently, $C_{\pi_{\Hil_{\Psi}}\Psi}$ is (bb), then $C_{\Psi}$ is (bb) and consequently $\Psi$ is a lower frame sequence in $\Hil$.



For the second part note that

$(a\Leftrightarrow b)$
 The sequence $\Psi$ is a Riesz-Fischer sequence in $\Hil$ if and only if $C_{\Psi}$ is surjective or equivalently $C_{\Psi}^r$ is surjective (\cite[Proposition 4.1]{xxlstoeant11}) if and only if ${{C}_{\Psi}^r}^*=\overline{D_{\Psi}^r}$
is (bb) by Proposition \ref{eq:surj}.

$(b\Leftrightarrow c)$ is Corollary \ref{eq:clD to psi}.

$(c\Leftrightarrow d)$
Applying \eqref{Cpisi:Csi}, $C_{\pi_{\Hil_{\Psi}}\Psi}$ is surjective if and only if $C_{\Psi}^r$ is surjective and so the claim follows immediately by \cite[Proposition 4.1]{xxlstoeant11}.


\end{proof}


An easy consequence of this corollary is one part of \cite[Theorem 3.2]{casoleli1}: Every complete Riesz-Fischer sequence is a minimal lower frame sequence. The other parts deal with independence,  so we can state
if $\Psi$ is a Riesz-Fischer sequence, then $\Psi$ is  minimal by definition of Riesz-Fischer sequences and also is $\omega$-independent.
As stated in \cite{casoleli1} if $D_{\Psi}$ is closed and surjective, then the converse  is true.
Later with our technique we can show in Corollary \ref{RC=LM}  that  it always holds.


\section{Sesquilinear form of sequences}

To treat the cases of lower frame sequence and Riesz-Fischer sequence further we now use a concept of generalized frame operators that was introduced in \cite{Corso2019} by applying representation theorems \cite{ka95-1}, for a sequence with closable synthesis operator and extended in \cite{Corso2019GeneralizedFO} for a generic sequence. 
We will extend this approach. 

\subsection{The Generalized Frame Operator} \label{sec:genframop1}

Assume that $\Psi$ is a sequence  and $D_{\Psi}$ is closable (equivalent to $C_\Psi$ being densely defined). Then the non-negative sesquilinear form 
$$\sum_{i\in I}\langle f,\psi_i\rangle \langle \psi_i,g\rangle, \qquad (f,g\in \domain{C_{\Psi}})$$
is closed, densely defined and symmetric.
Then using  Kato's representation theorem  (\cite[Theorem  VI.2.23]{ka95-1}), there is a positive and self-adjoint operator $\Gamma_{\Psi}$ such that $\domain{\Gamma_{\Psi}}\subseteq \domain{C_{\Psi}}$ and $$\Omega_{\Psi}(f,g)=\left\langle \Gamma_{\Psi}f,g\right\rangle, \qquad (f\in \domain{\Gamma_{\Psi}}, g\in \domain{C_{\Psi}}).$$
It is obvious to see that  $\Gamma_{\Psi}=C_{\Psi}^*C_{\Psi}.$ 
Note that $\domain{\Gamma_{\Psi}}$ is a core of $\Omega_{\Psi}$ (\cite[Theorem VI.2.1]{ka95-1})  and so it is dense in $\domain{C_{\Psi}}$.

Moreover,  $\Gamma_{\Psi}$ and $S_{\Psi}$ are equal whenever
 $\Psi$ is a Bessel sequence. Generally, the operator $\Gamma_{\Psi}$ is an extension of $S_{\Psi}$, i.e. $S_\Psi \subseteq \Gamma_{\Psi}$ and so it  is called \textit{generalized frame operator} of $\Psi$. 
If we 
assume that $S_{\Psi}$ is closable and densely defined, then
$$\overline{S}_{\Psi}=(D_{\Psi}C_{\Psi})^{**}\subseteq (C_{\Psi}^*D_{\Psi}^*)^*=\Gamma_{\Psi}^*=\Gamma_{\Psi}\qquad .$$
For more on this relations see Prop. \ref{cor:65}. 
\\

For arbitrary sequence $S_\Psi$ does not even have to be closable on $\Hil$,
 in while it is always closable on $\Hil_{\Psi}$ since $C_{\Psi}$ is densely defined
on $\Hil_{\Psi}$.
In the following, we present an example of a sequence such that its frame operator is not a closable operator on $\Hil$.
\begin{ex}
Let $\{e_i\}_{i\in I}$ be  an orthonormal basis for a Hilbert space $\Hil$. 
Consider the sequence 
$\Psi=\left\{(1+i)e_i\right\}_{i\in I}$.
For all $f\in \domain{S_{\Psi}}$,
$$S_{\Psi}f=\sum_{i\in I}(1+i)^2\left<f,e_i\right>e_i.$$
We can observe that $\left(\frac{1}{n^2}e_n\right)$ is a sequence in $\domain{S_{\Psi}}$ and $\frac{1}{n^2}e_n \to 0$, in while 
$S_{\Psi}\left(\frac{1}{n^2}e_n\right) \to 1$ and therefore $S_{\Psi}$ is not closable.
\end{ex}

It is known that in the above setting, i.e. if $D_\Psi$ is closable then $S_\Psi$ is also \cite[Proposition 3.3]{xxlstoeant11}. 
The operator $S_\Psi$ can be closable, even if $D_\Psi$ is not closable, see the next example:

\begin{ex} Let $\Psi = \left( e_1, e_2, e_1, e_3, e_1, e_4,  \dots \right)$, then $S_\Psi$ is closed (it is even bounded on its domain), but $\domain{C_\Psi} = \clsp{e_2,e_3,e_4, \dots} \not= \Hil$.

$$S_\Psi f = \sum \limits_{k = 2}^\infty  \left< f, e_1 \right> e_1 + \sum \limits_{k=1}^\infty \left< f, e_k\right> e_k = \sum \limits_{k = 2}^\infty  \left< f, e_1 \right> e_1 + f. $$
Therefore $S_\Psi$ is not well-defined on $\clsp{e_1}$ ($\domain{S_{\Psi}}\subseteq \domain{C_{\Psi}}$) and it is the identity on  $\clsp{e_2,e_3,e_4, \dots}$.

\end{ex}

We now want to go beyond this setting, and aim at the definition of a generalized frame operator for arbitrary sequences. 
In \cite{Corso2019GeneralizedFO}, the author showed that if $D_{\Psi}$ is not  closable, then we can restrict the Hilbert space $\Hil$ to $\Hil_{\Psi}$. 
Then $\Omega_{\Psi}$ is a sesquilinear form in $\Hil_{\Psi}$  fulfilling the above assumptions. 
Similarly, there is a positive self-adjoint operator $\Gamma_{\Psi}:\domain{\Gamma_{\Psi}}\subseteq \Hil_{\Psi}\rightarrow \Hil_{\Psi}$ that $\domain{\Gamma_{\Psi}^{1/2}}=\domain{C_{\Psi}}$ and 
$$\Omega_{\Psi}(f,g)=\left\langle \Gamma_{\Psi}^{1/2}f,  \Gamma_{\Psi}^{1/2}g\right\rangle, \qquad (f,g\in \domain{C_{\Psi}}),$$
where $\Gamma_{\Psi}^{1/2}:\domain{\Gamma_{\Psi}^{1/2}}\subseteq \Hil_{\Psi}\rightarrow \Hil_{\Psi}$ is positive and the square root of $\Gamma_{\Psi}$.  

Note that if $\Psi$ is a  sequence  in $\Hil$, then $\Gamma_{\Psi}$ is always a densely defined operator in $\Hil_{\Psi}$. Moreover, $\Gamma_{\Psi}$ is a densely defined operator in $\Hil$ if and only if  $D_{\Psi}$ is a closable operator. Indeed, if $C_{\Psi}$ is a densely defined operator then $C_{\Psi}$ and $C_{\Psi}^*$ are closed and densely defined operators and so  $\Gamma_{\Psi}=C_{\Psi}^*C_{\Psi}$ is  densely defined  by Proposition  \ref{4propertdens} (6). Conversely, by the Kato's second representation theorem $\domain{\Gamma_{\Psi}}\subseteq \domain{C_{\Psi}}$ and so, if $\Gamma_{\Psi}$ is seen densely defined then $C_{\Psi}$ is densely defined. 
Consequently, 
\begin{equation} \label{eq:restrictgamma1}
\Gamma_{\Psi}= C_{\Psi}^{\times}C_{\Psi}= \left(C_{\Psi}^r\right)^* C_{\Psi}^r= \Gamma_{\Psi}^{r}
\end{equation}
and so the "restricted generalized frame operator" $\Gamma_{\Psi}^r$ is equal to $\Gamma_{\Psi}$, and therefore densely defined in $\Hil_{\Psi}$.
Moreover, \eqref{Cpisi:Csi} follows that $$\pi_{\Hil_{\Psi}}\Gamma_{\Psi}\subseteq\pi_{\Hil_{\Psi}}\Gamma_{\Psi}\pi_{\Hil_{\Psi}}=\pi_{\Hil_{\Psi}}\Gamma_{\Psi}^r\pi_{\Hil_{\Psi}}\subseteq\Gamma_{\pi_{\Hil_{\Psi}}\Psi}.$$
Indeed, the first inclusion is pointed out of this fact that $\domain{\Gamma_{\Psi}}\subseteq \domain{C_{\Psi}}$. Also,
\begin{eqnarray*}
\Gamma_{\pi_{\Hil_{\Psi}}\Psi}&=& C_{\pi_{\Hil_{\Psi}}\Psi}^*C_{\pi_{\Hil_{\Psi}}\Psi}\\
&=& (C_{\Psi}\pi_{\Hil_{\Psi}})^*C_{\Psi}\pi_{\Hil_{\Psi}}\\
&\supseteq & \pi_{\Hil_{\Psi}}C_{\Psi}^*C_{\Psi}\pi_{\Hil_{\Psi}}\\
&=& \pi_{\Hil_{\Psi}}\Gamma_{\Psi}\pi_{\Hil_{\Psi}}=  \pi_{\Hil_{\Psi}}\Gamma_{\Psi}^r\pi_{\Hil_{\Psi}}.
\end{eqnarray*}
Also, from $\domain{\Gamma_{\pi_{\Hil_{\Psi}}\Psi}}\subseteq \domain{C_{\pi_{\Hil_{\Psi}}\Psi}}$ and  $\overline{\domain{C_{\pi_{\Hil_{\Psi}}\Psi}}}=\Hil$ it follows that $\Gamma_{\pi_{\Hil_{\Psi}}\Psi}$ is densely defined in $\Hil$.

Moreover, we can compare the generalized frame operator of a lower frame sequence with its related sequence.
\begin{proposition}\label{com:=Ga}
Let $\Psi\subseteq \Hil$. Then 
$
\Gamma_{\Psi}^r\subseteq\Gamma_{\pi_{\Hil_{\Psi}}\Psi}.
$
\end{proposition}
\begin{proof}
 Obviously,  $\domain{\Gamma_{\Psi}}\subseteq \domain{\Gamma_{\pi_{\Hil_{\Psi}}\Psi}}$ by using the property $C_{\pi_{\Hil_{\Psi}}\Psi}=C_{\Psi}\pi_{\Hil_{\Psi}}$ and then $\Gamma_{\pi_{\Hil_{\Psi}}\Psi}f=\Gamma_{\Psi}f$, for all $f\in \domain{\Gamma_{\Psi}}$.
\end{proof}

In the following, we prove that the frame operator of an arbitrary sequence is always closable.
\begin{corollary} \label{cor:65}
Let $\Psi\subseteq \Hil$ a sequence. 
Then $S_{\Psi}|_{\Hil_{\Psi}}$ is closable (with closure $C_{\Psi}^{\times}C_{\Psi}$) and 
\begin{eqnarray}\label{re:*barS}
S_{\Psi}\subseteq \overline{\left(S_{\Psi}|_{\Hil_{\Psi}}\right)}\subseteq \Gamma_{\Psi}\subseteq S_{\Psi}^{\times}.
\end{eqnarray}
\end{corollary}
\begin{corollary}
Let $\Psi$ be  a sequence in $\Hil$. 
The following are equivalent:
\begin{enumerate}
 \item $\Gamma_{\Psi}=\overline{\left({S}_{\Psi}|_{\Hil_{\Psi}}\right)}$;
    \item $\overline{\left({S}_{\Psi}|_{\Hil_{\Psi}}\right)}$ is relaxed self-adjoint;
    \item  $S_{\Psi}^{\times}$ is symmetric; 
    \item $S_{\Psi}^{\times}$ is relaxed self-adjoint;
    \item $\overline{\left({S}_{\Psi}|_{\Hil_{\Psi}}\right)}=S_{\Psi}^{\times}$.
\end{enumerate}
\end{corollary}
\begin{proof}
$(1\Leftrightarrow 2)$ 
Consider $\overline{\left({S}_{\Psi}|_{\Hil_{\Psi}}\right)}=\Gamma_{\Psi}$. Then 
 $$\overline{\left({S}_{\Psi}|_{\Hil_{\Psi}}\right)}^{\times}=\Gamma_{\Psi}^*=\Gamma_{\Psi}=\overline{\left({S}_{\Psi}|_{\Hil_{\Psi}}\right)}.$$
Conversely, 
assume that 
$\overline{\left({S}_{\Psi}|_{\Hil_{\Psi}}\right)}$ is relaxed self-adjoint. Using
$\overline{\left({S}_{\Psi}|_{\Hil_{\Psi}}\right)}\subseteq \Gamma_{\Psi}$ we get $\Gamma_{\Psi}=\Gamma_{\Psi}^*\subseteq \overline{\left({S}_{\Psi}|_{\Hil_{\Psi}}\right)}^{\times}=\overline{\left({S}_{\Psi}|_{\Hil_{\Psi}}\right)}.$
Also,
 from \eqref{re:*barS} it follows that
 $\Gamma_{\Psi}=\overline{\left({S}_{\Psi}|_{\Hil_{\Psi}}\right)}$.  
 $(2\Leftrightarrow 3)$ 
It is  straightforward to see $\overline{S_{\Psi}|_{\Hil_{\Psi}}}=S_{\Psi}^{\times \times}$. The rest is proved in  \cite[Problem V.3.10]{ka95-1}.
 
 and $(3\Leftrightarrow 4)$ are proved in  \cite[Problem V.3.10]{ka95-1}.
 
 $(4\Leftrightarrow 5)$ is followed by  \eqref{re:*barS}.
\end{proof}
\begin{corollary}
   Let $\Psi\subseteq \Hil$ be a sequence. 
   Then
\begin{eqnarray}\label{S:rel}
S_{\Psi}^r\subseteq \overline{S_{\Psi}^r}\subseteq \Gamma_{\Psi}^r\subseteq \left(S_{\Psi}^r\right)^{\times}\end{eqnarray}
\end{corollary}
 Applying \eqref{eq:restrictgamma1} to \eqref{re:*barS}
and \eqref{S:rel} we also get:
$$S_{\Psi}\subseteq \overline{\left(S_{\Psi}|_{\Hil_{\Psi}}\right)}\subseteq \Gamma_{\Psi}\subseteq \left(S_{\Psi}^r\right)^{\times}, \quad S_{\Psi}^r\subseteq \overline{S_{\Psi}^r}\subseteq \Gamma_{\Psi}^r\subseteq (S_{\Psi})^{\times}.$$

Let us remind the reader to an important property for generalized frame operators of  lower frame sequences.
\begin{lemma}\cite[Proposition 3.3]{Corso2019GeneralizedFO}\label{invertgeneral}
Let $\Psi$ be a sequence in $\Hil$. The following are equivalent:
\begin{enumerate}
    \item $\Psi$ is a lower frame sequence with lower frame bound $A$.
    \item $\Gamma_{\Psi}$ is bounded from below by $A$ (i.e. $\|\Gamma_{\Psi}f\|\geq A\|f\|$) for all $f\in \domain{C_{\Psi}}.$
    \item $\Gamma_{\Psi}$ is invertible and $\Gamma_{\Psi}^{-1}\in \BL{\Hil_{\Psi}}$ with $\|\Gamma_{\Psi}^{-1}\|\leq  A^{-1}$. (In particular, $\domain{\Gamma_{\Psi}^{-1}}=\range{\Gamma_{\Psi}}=\Hil_{\psi}$.)
    \end{enumerate}
\end{lemma}

 It can be easily seen that if $\Psi$ is a lower frame sequence,  $D_{\Psi}$ is closable and $\domain{\Gamma_{\Psi}}$ is  closed, then $\Gamma_{\Psi}$ is a bounded operator (by \cite[Theorem 5.5.6]{weidm80}). 
So, 
$$\Psi=\Gamma_{\Psi}\Gamma_{\Psi}^{-1}\Psi$$
 is a Bessel sequence in $\Hil_{\Psi}$ and in this case, it  is a frame for $\Hil_{\Psi}$.

Following the main topic of this manuscript we introduce more properties of the generalized frame operator of a sequence, based on the different notions of invertibility.
\begin{lemma}\label{surj=inj}
Let $\Psi$ be a sequence in $\Hil$. 
The following are equivalent:
\begin{enumerate}
    \item $\Gamma_{\Psi}$ is (bb).
    \item $\Gamma_{\Psi}$ is surjective onto $\Hil_{\Psi}$.
    \item $\Gamma_{\Psi}$ is injective and has closed range.
    \item $\Gamma_{\Psi}$ is (BI) (i.e.  $\Gamma_{\Psi}^{-1}\in \BL{\Hil_{\Psi}}$).
    \end{enumerate}
\end{lemma}

\begin{proof}
The operators $C_{\Psi}$ and $C_{\Psi}^r$ are closed and densely defined operators and so $\Gamma_{\Psi}= \left(C_{\Psi}^r\right)^*C_{\Psi}^r$ is a closed, densely defined and self-adjoint operator on $\Hil_{\Psi}$ by Proposition \ref{4propertdens} (6). 
Then it is 
surjective on $\Hil_{\Psi}$ if and only if it is (bb) (by Proposition \ref{eq:surj}) if and only if  is injective and has closed range (by Proposition \ref{prop:injcloran1}).
Therefore by Corollary \ref{eq:BI}, $\Gamma_\Psi$  is surjective if and only if it is (BI) and is equivalent to $(1)$ and $(2)$. 
\end{proof}
In \cite{Corso2019GeneralizedFO}, it is stated the sequence $\Gamma_{\Psi}^{-1}\pi_{\Hil_{\Psi}}\Psi \in \Hil_{\Psi}$ is a Bessel sequence, whenever $\Psi$ is a lower frame sequence in $\Hil$. Furthermore, this was used to extend the concept of  the  canonical dual of any lower frame sequence: The sequence $\Gamma_{\Psi}^{-1}\pi_{\Hil_{\Psi}}\Psi$ is called the \emph{the canonical dual } of a lower frame sequence  $\Psi$ and denoted by $\wiSi$ \cite{Corso2019GeneralizedFO}. It fulfils $f = \sum \limits_{k \in I}  \left< f, \pi_{\Hil_{\Psi}} \psi_k \right> \tilde \psi_k$ for all $f \in \domain{C_\Psi}$. 
In particular, it is 
easy to see that if $\Psi$ is  a lower frame sequence in $\Hil$ such that $D_{\Psi}$ is a closable operator, then $\wiSi$ is $\Gamma_{\Psi}^{-1}\Psi$. 

It is crucial to know a formula for the inverse of the generalized frame operator of a lower frame sequence to present a reconstruction formula. 
\begin{proposition}\label{rec:form}
Let $\Psi$ be a lower frame sequence. Then $\Gamma_{\Psi}$ is invertible on $\Hil_{\Psi}$ and 
\begin{eqnarray}\label{invTgeneral}
\Gamma_{\Psi}^{-1}\pi_{\Hil_{\Psi}}f=\sum_{i\in I} \left< f, \Gamma_{\Psi}^{-1}\pi_{\Hil_{\Psi}}\psi_i\right>\Gamma_{\Psi}^{-1}\pi_{\Hil_{\Psi}}\psi_i, \qquad (f\in \Hil)
\end{eqnarray}
with unconditional convergent.
\end{proposition}
\begin{proof}
 $\Gamma_{\Psi}^{-1}\pi_{\Hil_{\Psi}}f\in \domain{\Gamma_{\Psi}}\subseteq \domain{C_{\Psi}}$, then for all $h\in \Hil$ we have
\begin{eqnarray*}
\left< h, \Gamma_{\Psi}^{-1}\pi_{\Hil_{\Psi}}f\right>&=& \left< \pi_{\Hil_{\Psi}}h, \Gamma_{\Psi}^{-1}\pi_{\Hil_{\Psi}}f\right>\\
&=& \left< \Gamma_{\Psi}\Gamma_{\Psi}^{-1}\pi_{\Hil_{\Psi}}h, \Gamma_{\Psi}^{-1}\pi_{\Hil_{\Psi}}f\right>\\
&=& \sum_{i\in I} \left< \Gamma_{\Psi}^{-1}\pi_{\Hil_{\Psi}}h,\psi_i\right> \left< \psi_i, \Gamma_{\Psi}^{-1}\pi_{\Hil_{\Psi}}f\right>\\
&=& \sum_{i\in I} \left< \pi_{\Hil_{\Psi}}h,\Gamma_{\Psi}^{-1}\pi_{\Hil_{\Psi}}\psi_i\right> \left< \Gamma_{\Psi}^{-1}\pi_{\Hil_{\Psi}}\psi_i, f\right>\\
&=& \sum_{i\in I} \left< h,\Gamma_{\Psi}^{-1}\pi_{\Hil_{\Psi}}\psi_i\right> \left< \Gamma_{\Psi}^{-1}\pi_{\Hil_{\Psi}}\psi_i, f\right>= \left< h, \sum_{i\in I} \left< f, \Gamma_{\Psi}^{-1}\pi_{\Hil_{\Psi}}\psi_i\right>\Gamma_{\Psi}^{-1}\pi_{\Hil_{\Psi}}\psi_i\right>.
\end{eqnarray*}
By the convergence of $\sum_{i\in I}\left<f,\wisi_i\right> \wisi_i$, for all $f\in \Hil$, we conclude
\begin{eqnarray*}
\Gamma_{\Psi}^{-1}\pi_{\Hil_{\Psi}}f=\sum_{i\in I} \left< f, \Gamma_{\Psi}^{-1}\pi_{\Hil_{\Psi}}\psi_i\right>\Gamma_{\Psi}^{-1}\pi_{\Hil_{\Psi}}\psi_i, \qquad (f\in \Hil).
\end{eqnarray*}
\end{proof}
Using \cite[Theorem 4.1]{Corso2019GeneralizedFO} we have
\begin{eqnarray}\label{reconstbyT}
f=\sum_{i\in I} \left< f, \pi_{\Hil_{\Psi}}\psi_i\right>\wisi_i, \qquad (f\in \domain{C_{\Psi}}).
\end{eqnarray}

We will now ask ourselves the question what happens for the other duality direction:
\begin{corollary}\label{iffgama}  
Let $\Psi$ be a lower frame sequence in $\Hil$. Then 
\begin{eqnarray}\label{symdual}
f=\sum_{i\in I}\left< f, \Gamma_{\Psi}^{-1}\pi_{\Hil_{\Psi}}\psi_i\right> \pi_{\Hil_{\Psi}}\psi_i, \qquad (f\in \domain{\Gamma_{\Psi}}),
\end{eqnarray}
if and only if 
$C_{\Psi}(\Gamma_{\Psi}^{-1}f)\in \domain{D_{\Psi}^r}$.

\end{corollary}
\begin{proof}
Suppose \eqref{symdual} is fulfilled. 
Since $\Gamma_{\Psi}^{-1}f\in \domain{\Gamma_\Psi}$, for all $f\in \Hil_{\Psi}$, then by assumption $\sum_{i\in I}\left< f, \Gamma_{\Psi}^{-1}\pi_{\Hil_{\Psi}}\psi_i\right> \pi_{\Hil_{\Psi}}\psi_i$ converges and so,
$D_{\Psi}^rC_{\Psi} ^r(\Gamma_{\Psi}^{-1}f)$ converges. 

Conversely,
using \eqref{invTgeneral} , for $f \in \domain{\Gamma_\Psi}$ we have
$$ 
f=\Gamma_{\Psi} \sum_{i\in I} \left< f, \Gamma_{\Psi}^{-1}\pi_{\Hil_{\Psi}}\psi_i\right>\Gamma_{\Psi}^{-1}\pi_{\Hil_{\Psi}}\psi_i. $$
Because by assumption the right hand side is well-defined we have 
$$ f =  \sum_{i\in I} \left< f, \Gamma_{\Psi}^{-1}\pi_{\Hil_{\Psi}}\psi_i\right>\pi_{\Hil_{\Psi}}\psi_i.
$$
\end{proof}
It is easy to see that if $\Psi$ is a lower frame sequence and $D_{\Psi}$ is closable then by 
\eqref{reconstbyT} we  have
\begin{eqnarray*}
f=\sum_{i\in I} \left< f, \psi_i\right> \Gamma_{\Psi}^{-1}\psi_i, \qquad (f\in \domain{\Gamma_{\Psi}}).
\end{eqnarray*}
By Corollary \ref{iffgama}, if $D_{\Psi}$ is a closable operator then we have
$$f=\sum_{i\in I}\left<f, \Gamma_{\Psi}^{-1}\psi_i\right> \psi_i, \qquad (f\in \domain{\Gamma_{\Psi}})$$
if and only if $\left\{\left< f, \Gamma_{\Psi}^{-1}\psi_i\right>\right\}_{i\in I}\in \domain{D_{\Psi}}$.

 As we mentioned, for every lower frame sequence $\Psi$ the sequence $\wiSi$ is a Bessel sequence in  $\Hil$ and so we can show that $\Gamma_{\wiSi}=S_{\wiSi}=\Gamma_{\Psi}^{-1}$ on $\Hil_{\Psi}$. And so, 
$\domain{S_{\wiSi}}=\Hil_{\Psi}$ \cite[Proposition 3]{Corso2019}, and hence
for every $f\in \Hil_{\Psi}$ we have
\begin{eqnarray*}
\Gamma_{\wiSi}f&=& S_{\wiSi}f\\
&=&  \sum_{i\in I} \left< f, \Gamma_{\Psi}^{-1}\pi_{\Hil_{\Psi}}\psi_i\right>\Gamma_{\Psi}^{-1}\pi_{\Hil_{\Psi}}\psi_i=\Gamma_{\Psi}^{-1}f.
\end{eqnarray*}
This shows that $\widetilde{\Psi}$ is an upper semi frame for $\Hil_{\Psi}$.

Now we are ready to address the pseudo-inverse of $C_{\Psi}$ and state it by effective expression.
 \begin{proposition}\label{pseudoC}
Assume that 
$\Psi\subseteq \Hil$ is a lower frame sequence. Then $C_{\Psi}$ has the bounded 
pseudo-inverse $C_{\Psi}^{\dag}=\Gamma_{\Psi}^{-1}{C}_{\Psi}^r$. In particular, $C_{\Psi}^{\dag}=D_{\wiSi}$ on $\domain{{C}_{\Psi}^r}$.
\end{proposition}

\begin{proof}
It is straightforward to see that
$C_{\Psi}\Gamma_{\Psi}^{-1}{C}_{\Psi}^r{C}_{\Psi}=C_{\Psi}$. Also, by the surjectivity ${C}_{\Psi}^r$ on $\Hil_{\Psi}$  (Corollary \ref{eq:clD to psi}) and injectivity $C_{\Psi}$  (\cite[Lemma 3.1]{casoleli1})  follow that
\begin{eqnarray*}
\overline{\range{\Gamma_{\Psi}^{-1}{C}_{\Psi}^r}}=\overline{\domain{C_{\Psi}}}=\Hil_{\Psi}=\kernel{C_{\Psi}}^{\perp}.
\end{eqnarray*}
Using Proposition \ref{4propertdens} (4) we have 
\begin{eqnarray*}
\kernel{\Gamma_{\Psi}^{-1}{C}_{\Psi}^r}=\kernel{{C}_{\Psi}^r}=\range{C_{\Psi}}^{\perp}.
\end{eqnarray*}
So, $C_{\Psi}^{\dag}=\Gamma_{\Psi}^{-1}{C}_{\Psi}^r$. 
Also, for all $g\in \Hil$ and $\{c_i\}_{i\in I}\in \domain{{C}_{\Psi}^{r}}$ we have
\begin{eqnarray*}
\left< \Gamma_{\Psi}^{-1}{C}_{\Psi}^{r}\{c_i\}_{i\in I}, g\right>&=&\left< {C}_{\Psi}^{r}\{c_i\}_{i\in I}, \Gamma_{\Psi}^{-1}\pi_{\Hil_{\Psi}}g\right>\\
&=&\left< \{c_i\}_{i\in I}, C_{\Psi}\Gamma_{\Psi}^{-1}\pi_{\Hil_{\Psi}}g\right>\\
&=&\sum_{i\in I} \left<c_i, \left<  \Gamma_{\Psi}^{-1}\pi_{\Hil_{\Psi}}g, \psi_i\right>\right>\\
&=&\sum_{i\in I}c_i\left< \Gamma_{\Psi}^{-1}\pi_{\Hil_{\Psi}}\psi_i, g\right>\\
&=&\left<\sum_{i\in I}c_i\Gamma_{\Psi}^{-1}\pi_{\Hil_{\Psi}}\psi_i
,g\right>=\left<D_{\wiSi}\{c_i\}_{i\in I},g \right>.
\end{eqnarray*}
\end{proof}

\subsection{Generalized Gram Matrix}
Another sesquilinear form \cite{Corso2019} can be presented for sequences with the following argument. Assume that $\Psi$ is a sequence.   Then there is a  sesquilinear form
\begin{eqnarray*}
\Theta_{\Psi}(\{c_i\}_{i\in I}, \{d_i\}_{i\in I})=\sum_{i,j\in I}c_i\overline{d_j}\langle \psi_i,\psi_j\rangle, \qquad (\{c_i\}_{i\in I},\{d_i\}_{i\in I}\in \domain{D_{\Psi}})
\end{eqnarray*}
 on $\domain{D_{\Psi}}\times \domain{D_{\Psi}}$ and
$$\Theta_{\Psi}(\{c_i\}_{i\in I}, \{d_i\}_{i\in I})=\left<D_{\Psi}\{c_i\}_{i\in I}, D_{\Psi}\{d_i\}_{i\in I}\right>.$$

If $\Psi$ is a Riesz-Fischer sequence and $D_{\Psi}$ is a closable operator then 
the closure $\overline{\Theta}_{\Psi}$ of $\Theta_{\Psi}$ is a  non-negative, symmetric, closed and densely defined sesquilinear form and  
 $\domain{\overline{\Theta}_{\Psi}}=\domain{C_{\Psi}^*}$ \cite[Proposition 4]{Corso2019}. By Kato's second representation theorem for sesquilinear forms in this case there is a positive and self-adjoint operator $U_{\Psi}$ such that $\domain{U_{\Psi}}\subseteq \domain{C_{\Psi}^*}$  is dense in $\domain{C_{\Psi}^*}$ and 
 $$\overline{\Theta}_{\Psi}(\{c_i\}, \{d_i\})=\left\langle U_{\Psi}\{c_i\}, \{d_i\}\right\rangle,\qquad (\{c_i\}\in \domain{U_{\Psi}},  \{d_i\}\in \domain{C_{\Psi}^*}).$$
As in \cite[Proposition 4]{Corso2019} we have $U_{\Psi}=C_{\Psi}C_{\Psi}^*=C_{\Psi}\overline{D}_{\Psi}$. We call the operator $U_{\Psi}$, the \textit{generalized Gram matrix}. 

For every sequence $\Psi$ with  closable synthesis operator, we have 
 \begin{eqnarray}\label{Uc,d}
 \left<  U_{\Psi}c,d\right>= \sum_{i,j\in I}  c_i\overline{d_j}\left<\psi_i,\psi_j\right>, \qquad (c\in \domain{U_{\Psi}}, d\in \domain{\overline{D}_{\Psi}})
 \end{eqnarray}
 if  $\psi_i\in \domain{C_{\Psi}}$ 
 for all $i\in I$. Indeed, assume that  $\psi_i\in \domain{C_{\Psi}}$. Then
 \begin{eqnarray*}
 \left< U_{\Psi}c,d\right>&=& \left< C_{\Psi}C_{\Psi}^*c,d\right>\\
 &=&\left< \left\{\left<C_{\Psi}^*c,\psi_j\right>\right\}_{j},\{d_j\}_j\right>\\
 &=&\sum_{j\in I}\overline{d_j}\left<C_{\Psi}^*c,\psi_j\right>\\
 &=&\sum_{j\in I}\overline{d_j}\left<c,C_{\Psi}\psi_j\right>\\
 &=& \sum_{j\in I}\overline{d_j}\left<\{c_i\}_i,\left\{\left<\psi_j,\psi_i\right>\right\}_{i}\right>=\sum_{i,j\in I}  c_i\overline{d_j}\left<\psi_i,\psi_j\right>,
 \end{eqnarray*}
 for all $c\in \domain{U_{\Psi}}, d\in \domain{\overline{D}_{\Psi}}$.

 One may ask how can we extend the consideration for $D_{\Psi}$ not closable? 
As above ${C}_{\Psi}:\domain{C_{\Psi}}\subseteq \Hil_{\Psi}\to \ell^2$ is a closed densely defined operator and so ${C}_{\Psi}^{r}$ is. 
Let $\Theta_{\Psi}^r: \domain{D_{\Psi}^r}\times \domain{D_{\Psi}^r}\to \ell^2$ given by 
$\Theta_{\Psi}^r(c,d)=\left<  D_{\Psi}^r c, D_{\Psi}^r d\right>$. It 
is a  densely defined and  closable sesquilinear form with the closure 
$\overline{\Theta_{\Psi}^r}:\domain{\overline{{D}_{\Psi}^r}}\times \domain{\overline{{D}_{\Psi}^r}}\to \ell^2$ defined by 
$\overline{\Theta}_{\Psi}^r(c,d) = \left< \overline{D_\Psi^r} c, \overline{D_\Psi^r} d \right>$. 
Now, we can talk about the operator $U_{\Psi}^r$, where $U_{\Psi}^r: \domain{U_{\Psi}^r}\subseteq \ell^2 \to \ell^2$ and $\domain{U_{\Psi}^r}$ is dense in $\domain{\overline{{D}_{\Psi}^r}}$. Moreover, 
 $\overline{\Theta}_{\Psi}$ is a closed, densely defined, and non-negative sesquilinear form and so there is a positive and self-adjoint operator $U_{\Psi}$ in $\ell^2$ such that $\domain{U_{\Psi}^{1/2}}=\domain{\overline{{D}_{\Psi}^r}}$ and 
$$\overline{\Theta}_{\Psi}(c,d)=\left< U_{\Psi}^{1/2}c, U_{\Psi}^{1/2}d\right>, \qquad (c,d \in \domain{\overline{{D}_{\Psi}^r}}).$$
 This shows that for all $c,d \in \domain{\overline{D_{\Psi}^r}}$,
$$\overline{\Theta_{\Psi}^r}(c,d)=\left<\overline{D_{\Psi}^r} c, \overline{D_{\Psi}^r} d\right> =\overline{\Theta}_{\Psi}(c,d).$$
\begin{figure}[ht]
\center
	\begin{picture}(100,60)
	\put(10,50){$\Hil_{\Psi}$}
	\put(90,50){$\ell^2$}
	\put(87,52){\vector(-4,0){72}}
	\put(45,55){$D_{\Psi}^{r}$} 
	\put(38,32){$C_{\Psi}^r$} 
	\put(13,47){\vector(0,-4){30}}
	\put(91,17){\vector(0,4){30}}
   \put(15,14){\vector(2,1){70}}
	\put(11,10){$\Hil_{\Psi}$}
	\put(4,30){$\Gamma_{\Psi}^r$}
	\put(91,10){$\ell^2$}
	\put(84,30){$U_{\Psi}$}

	\put(15,12){\vector(4,0){72}}
	\put(50,7){$C_{\Psi}^r$} 
	\end{picture}
\caption{The generalized frame operator and the generalized Gram matrix} \label{fig:matop1}
\end{figure}
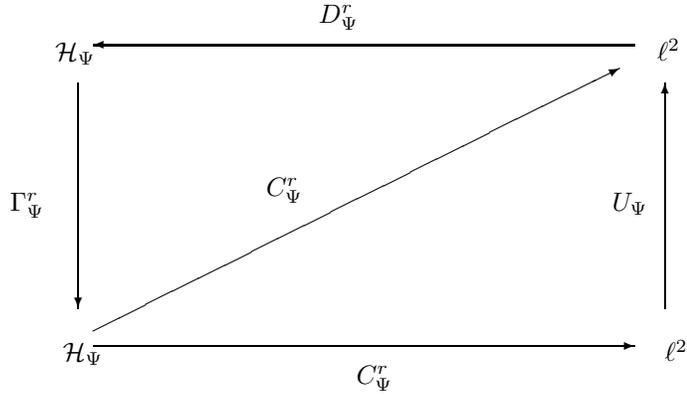

 \begin{lemma}\label{tetaRi}
 Let $\Psi$ be a  sequence in $\Hil$. 
 Then the following are equivalent.
 \begin{enumerate}
     \item $\Psi$ is a Riesz-Fischer sequence.
     \item $U_{\Psi}^r$ is invertible (BI) and $\left(U_{
     \Psi}^r\right)^{-1}\in \BL{\ell^2}$.
 \end{enumerate}
 \end{lemma}
 \begin{proof}
First note that if $\Psi$ is a sequence in $\Hil$, then 
$D_{\Psi}^r$
is a densely defined and closable operator (since $C_{\Psi}^r$ is a densely defined operator in $\Hil_{\Psi}$). Then 
$D_{\Psi}^r$
is (bb) if and only if  $\overline{D_{\Psi}^r}
$ is (bb) (using Lemma \ref{biclos})). It can be seen that
$\Psi$ is a Riesz-Fischer sequence if and only if  $\Theta_{\Psi}^r$ 
(bb) with positive lower bound $A$ if and only if  $\overline{\Theta_{\Psi}^r}$ is (bb) with positive lower bound $A$ (using Lemma \ref{biclos}) if and only if  $\|U_{\Psi}^{1/2}c\|^2=\overline{\Theta}_{\Psi}(c,c)\geq A\|c\|^2$ if and only if $\|U_{\Psi}^{-1}\|\leq A$.
\end{proof}

 Assume that $\Psi$ is a Riesz-Fischer sequence such that $\psi_i\in \domain{C_{\Psi}}$, for all $i\in I$ and $D_{\Psi}$ is a closable operator. Then for all $c\in \ell^2$ and $d\in \domain{\overline{D}_{\Psi}}$ we get
 \begin{eqnarray*}
 \left< c,d\right>=
 \left<U_{\Psi}U_{\Psi}^{-1}c,d\right>
 =\sum_{i,j\in I} \left( U_{\Psi}^{-1}c\right)_i\overline{d_j}\left<\psi_i,\psi_j\right>.
 \end{eqnarray*}

\subsection{Lower frame  and minimal sequences}
In this section - maybe surprisingly - we can now show that a lot of results that are known for frames are also valid for lower frame sequence (with some twist). 

For a lower frame sequence $\Psi$,  
we get
\begin{eqnarray}\label{weakgamma}\left<h,g\right>=\sum_{i\in I} \left<h, \wisi_i\right> \left<\psi_i,g\right>,\end{eqnarray}
for all $h\in \Hil$ and $g\in \domain{C_{\Psi}}$. Indeed,
\begin{eqnarray*}
 \sum_{i\in I} \left<h, \wisi_i\right> \left<\psi_i,g\right>&=&\sum_{i\in I} \left<\Gamma_{\Psi}^{-1}\pi_{\Hil_{\Psi}}h, \psi_i\right> \left<\psi_i,g\right>\\
&=& \left<\Gamma_{\Psi}\Gamma_{\Psi}^{-1}\pi_{\Hil_{\Psi}}h,g \right>\\
&=&\left\langle \pi_{\Hil_{\Psi}}h,g \right>= \left<h,g\right>
.
\end{eqnarray*}

In the following, we show that for a lower frame sequence $\Psi$, the coefficients $\left\{\left<f, \wisi_i \right>\right\}_{i\in I}$ have a minimal $\ell^2$-norm  among all sequences represents $h\in \Hil$ in a weak sense.
\begin{lemma}\label{minnorm}
Let $\Psi$ be a lower frame  sequence in $\Hil$ 
and $h\in \Hil$. 
Let $h$ be represented weakly by $\psi_i$, i.e. 
there is some coefficients $\{c_i\}_{i\in I}$ such that  $\left<h,g\right>=\left<\sum_{i\in I}c_i\psi_i,g\right>$, for all $g\in \domain{C_{\Psi}}$. Then 
$$\sum_{i\in I}|c_i|^2=\sum_{i\in I}\left|\left< h,\wisi_i\right>\right|^2+\sum_{i\in I}\left|c_i-\left< h,\wisi_i\right>\right|^2.$$
 In particular, $\sum_{i\in I}|c_i|^2\geq \sum_{i\in I}\left|\left< h,\wisi_i\right>\right|^2$ and so the canonical dual gives minimum norm. 
\end{lemma}
\begin{proof}
At first, note that for all $h\in \Hil$ we get
\begin{eqnarray*}
\left< h,\Gamma_{\Psi}^{-1}\pi_{\Hil_{\Psi}}h\right>&=&
\left< \sum_{i\in I}c_i\psi_i,\Gamma_{\Psi}^{-1}\pi_{\Hil_{\Psi}}h\right>\\
&=&\sum_{i\in I}c_i\left< \psi_i,\Gamma_{\Psi}^{-1}\pi_{\Hil_{\Psi}}h\right>= \sum_{i\in I}c_i\left< \Gamma_{\Psi}^{-1}\pi_{\Hil_{\Psi}}\psi_i,h\right>.
\end{eqnarray*}
 On the other hand, 

\begin{eqnarray*}
\sum_{i\in I}\left|\left< h,\wisi_i\right>\right|^2&=& \sum_{i\in I}\left|\left< \Gamma_{\Psi}^{-1}\pi_{\Hil_{\Psi}}h,\psi_i\right>\right|^2\\
&=&\sum_{i\in I}\left< \Gamma_{\Psi}^{-1}\pi_{\Hil_{\Psi}}h,\psi_i\right>\left< \psi_i,\Gamma_{\Psi}^{-1}\pi_{\Hil_{\Psi}}h\right>\\
&=&\left< \Gamma_{\Psi}\Gamma_{\Psi}^{-1}\pi_{\Hil_{\Psi}}h,\Gamma_{\Psi}^{-1}\pi_{\Hil_{\Psi}}h\right>= \left< h,\Gamma_{\Psi}^{-1}\pi_{\Hil_{\Psi}}h\right>.
\end{eqnarray*}
So,  for all  $h\in \Hil$ we have
 \begin{eqnarray*}
\sum_{i\in I}\left|c_i-\left< h,\wisi_i\right>\right|^2&=&\sum_{i\in I} |c_i|^2  + \sum_{i\in I}\left|\left< h,\wisi_i\right>\right|^2 - \sum_{i\in I}\overline{c_i}\left< h, \Gamma_{\Psi}^{-1}\pi_{\Hil_{\Psi}}\psi_i\right>  - \sum_{i\in I}c_i\left< \Gamma_{\Psi}^{-1}\pi_{\Hil_{\Psi}}\psi_i,h\right>\\
&=& \sum_{i\in I} |c_i|^2 + \left< h,\Gamma_{\Psi}^{-1}\pi_{\Hil_{\Psi}}h\right> - \left<\Gamma_{\Psi}^{-1}\pi_{\Hil_{\Psi}} h,h\right>  - \left< h,\Gamma_{\Psi}^{-1}\pi_{\Hil_{\Psi}}h\right>\\
&=& \sum_{i\in I} |c_i|^2 - \left< \Gamma_{\Psi}^{-1}\pi_{\Hil_{\Psi}}h,h\right>\\
&=& \sum_{i\in I} |c_i|^2- \sum_{i\in I}\left|\left< h,\wisi_i\right>\right|^2.
\end{eqnarray*}
\end{proof}
 By Corollary \ref{pseudogeneral} and Lemma \ref{minnorm}, we can present an explicit expression for the pseudo-inverse of $\overline{D}_{\Psi}$, where  $\Psi$ is a lower frame sequence with closable synthesis operator.
\begin{theorem}
Let $\Psi$ be a lower frame sequence in $\Hil$. 
Then 
$$\overline{D_{\Psi}^r}^{\dag}f=\left\{\left<f,\wisi_i\right>\right\}_{i\in I}, \qquad (f\in \Hil_{\Psi}).$$
In particular, if $D_{\Psi}$ is closable, then $\overline{D}_{\Psi}^{\dag}=\left\{\left<f,\wisi_i\right>\right\}_{i\in I}$.
\end{theorem}
\begin{proof}
It immediately follows from \eqref{weakgamma} and Lemma \ref{minnorm}.
\end{proof}
In the following result, we present some enough conditions for exact lower frame sequences.
\begin{theorem}\label{ex=inco}
Suppose that $\Psi$ is a lower frame sequence in $\Hil$. Then for any $j\in I$
\begin{enumerate}
    \item $\left<\psi_j, \wisi_j\right>\neq 1$, then $\{\psi_{k}\}_{k\neq j}$ is a lower frame sequence for $\Hil$.
    \item $\left<\psi_j, \wisi_j\right>= 1$, then $\{\psi_{k}\}_{k\neq j}$ is incomplete.
\end{enumerate}
\end{theorem}
\begin{proof}
Similar to \cite[Theorem 5.4.7]{ole1n}, assume that $a_k=\left<\psi_j, \wisi_k\right>$. Then  $\psi_j=\sum_{k\in I}\left< \psi_j, \wisi_k\right>\psi_k$, in the weak sense. On the other hand $\psi_j=\sum_{i\in I}\delta_{j,k}\psi_k$ in the  weak sense. By Lemma \ref{minnorm}, we have 
\begin{eqnarray}\label{aj=1}
1=|a_j|^2+\sum_{k\neq j}|a_k|^2+|a_j-1|^2+\sum_{k\neq j}|a_k|^2.
\end{eqnarray}
Assume that $a_j=1$ then by \eqref{aj=1} we have $\sum_{k\neq j}|a_k|^2= 0$ and so $\left<\psi_j, \wisi_k\right>=0$, for all $k\neq j$. Then $\widetilde{\psi}_j\in {\overline{\text{span}}\{\psi_k\}}^{\perp}$ and so $\wisi_j\neq 0$ (as
 $\left<\psi_j, \wisi_j\right>=1$, for all $j\in I$). Therefore, $\{\psi_k\}_{k\neq j}$ is incomplete.
 
 Now assume that $a_j\neq 1$. Notice  $g\in \domain{C_{\Psi}}$ if and only if  $g\in \domain{C_{\Psi ^{(j)}}}$, where $\Psi^{(j)}=\{\psi_k\}_{k\neq j}$. Then for all $g\in \domain{C_{\Psi}}$ we get
 \begin{eqnarray*}
   \left|\left<g,\psi_j\right>\right|^2&=&\left|\left<g,\frac{1}{1-a_j}\sum_{k\neq j}a_k\psi_k\right>\right|^2\leq M\sum_{k\neq j}\left|\left<g,\psi_k\right>\right|^2,
 \end{eqnarray*}
 where $M=\frac{1}{|1-a_j|^2}\sum_{k\neq j}|a_k|^2$. Then
 \begin{eqnarray*}
 \sum_{k\neq j}\left|\left<g,\psi_k\right>\right|^2&\geq& \frac{1}{M+1}\left(\sum_{k\neq j}\left|\left<g,\psi_k\right>\right|^2+\left|\left<g,\psi_j\right>\right|^2\right)\\
 &=& \frac{1}{M+1}\sum_{k\in I}\left|\left<g,\psi_k\right>\right|^2\geq \frac{A}{1+M}\|g\|^2,
 \end{eqnarray*}
 where $A$ is a lower bound of $\Psi$.    
Hence, $\Psi^{(j)}$ is a lower frame sequence for $\Hil$.

\end{proof}
\begin{corollary}\label{ex=bio}
If $\Psi$ is an exact lower frame sequence, then $\psi$ has a (unique) biorthogonal sequence and so is minimal.
\end{corollary}
\begin{proof}
Assume that $\Psi$  is an exact lower frame sequence in $\Hil$, then $\left<\psi_j, \wisi_j\right>= 1$ by Theorem    
\ref{ex=inco} and also the proof shows that $\left<\psi_j, \wisi_k\right>= \delta_{j,k}$ that leads to the minimality of $\Psi$ by \cite[Lemma 3.3.1]{ole1n}.  Also, by \cite[Lemma 3.3.1]{ole1n}, the biorthogonal sequence is unique because $\Psi$ is a complete sequence.  
\end{proof}
\begin{corollary}
Let $\Psi$ be  an exact lower frame  sequence in $\Hil$. Then $\wiSi=\widetilde{\pi_{\Hil_{\Psi}}\Psi}$.
\end{corollary}
\begin{proof}
 Using Proposition \ref{com:=Ga}, $\Gamma_{\pi_{\Hil_{\Psi}}\Psi}^{-1}=\left(\Gamma_{\Psi}^r\right)^{-1}$ on $\range{\Gamma_{\Psi}^r}$. So, it is enough to show that $\pi_{\Hil_{\Psi}}\Psi \in \range{\Gamma_{\Psi}^r}$. For all $i\in I$,
 $$\wisi_i=\left(\Gamma_{\Psi}^r\right)^{-1}\pi_{\Hil_{\Psi}}\psi_i\in \domain{\Gamma_{\Psi}^r}\subseteq \domain{C_{\Psi}}=\domain{C_{\Psi}}$$
  and applying Corollary \ref{ex=bio}, $\left<\psi_i,\wisi_j\right>=\delta_{ij}$, for all $i,j\in I$. Then 
  $$D_{\Psi}^r C_{\Psi}^r \wisi_i=D_{\Psi}^r \delta_{ij}=\pi_{\Hil_{\Psi}}\psi_i$$
   and therefore $\pi_{\Hil_{\Psi}}\Psi\in \range{\Gamma_{\Psi}^r}$.
\end{proof}

In the following we state  that a minimal sequence is a Riesz-Fischer sequence when the synthesis operator is closable and surjective.
\begin{proposition}\label{minRiesz}
Let $\Psi$ be a sequence in $\Hil$. 
The following are equivalent.
\begin{enumerate}
    \item $\Psi$ is a Riesz-Fischer sequence. 
    \item  $\Psi$ is minimal   and $D_{\Psi}^r$ has closed range.
\end{enumerate} 

\end{proposition}
\begin{proof}
$(1)\Rightarrow (2)$ is proved in \cite[Theorem 3.2]{casoleli1} and Corollary \ref{eq:clD to psi}. \\
$(2)\Rightarrow (1)$ 
 Let $\Psi$ be a minimal sequence then it has a biorthogonal sequence 
 $\{g_j\}_{j\in I}$ and for all $j\in I$ we get
\begin{eqnarray}\label{biodomC}
e_j=\left\{\left< g_j,\psi_i\right>\right\}_{i\in I}=C_{\Psi}g_j
\end{eqnarray}
and then $g_j\in \domain{C_{\Psi}}$. Also,
$\overline{D_{\Psi}^r}$ is injective. Indeed, if $c\in \domain{\overline{D_{\Psi}^r}}$ and $\overline{D_{\Psi}^r}c=0$. Then for all $j\in I$ we obtain
\begin{eqnarray*}
0&=& \left< \overline{D_{\Psi}^r}c, g_j\right>\\
&=& \left< \{c_i\}_{i\in I} , C_{\Psi}g_j\right>\\
&=&\left< \{c_i\}_{i\in I}, \left\{\left<g_j,\psi_i\right>\right\}_{i\in I}\right>\\
&=& \sum_{i\in I}\left< c_i, \left<g_j,\psi_i\right>\right>=c_j
\end{eqnarray*}
and so $c_j=0$ for all $j\in I$. Hence by Corollary \ref{eq:clD to psi}, $\Psi$ is a Riesz-Fischer sequence.

\end{proof}
Now, we are ready to show that every minimal lower frame sequence is a complete Riesz-Fischer sequence in $\Hil$.
 \begin{corollary}\label{RC=LM}
Every minimal lower frame sequence in $\Hil$ is a complete Riesz-Fischer sequence in $\Hil$.
 \end{corollary}
\begin{proof}
Assume that $\Psi$ is a minimal lower frame sequence in $\Hil$. It is enough to show that $\Psi$ is a Riesz-Fischer sequence in $\Hil$ since every lower frame sequence is complete.  By the proof of Proposition \ref{minRiesz}, $\overline{D_{\Psi}^r}$ is injective. Moreover,  Corollary \ref{mathDlowriesz} shows that $\overline{D_{\Psi}^r}$ is surjective and so it has closed range. 
Again Corollary \ref{mathDlowriesz} demonstrates $\Psi$ is a Riesz-Fischer sequence in $\Hil$.
\end{proof}

Also, by \cite[Proposition 4.1 (e)]{xxlstoeant11} we have
\begin{proposition}
Let $\Psi$ be a sequence in $\Hil$. The following are equivalent:
\begin{enumerate}
    \item $\Psi$ is a lower frame sequence in $\Hil$.
    \item $\Psi$ is  complete and $C_{\Psi}$ has closed range.
\end{enumerate}
\end{proposition}

 In the following, we present some equivalent conditions for a lower frame sequence be a complete Riesz-Fischer sequence. 
 This gives a full generalization to the non-Bessel case of the well-known result relating frames to Riesz basis \cite[Theorem 6.1.1]{ole1n}:
 \begin{corollary}\label{eqRicompllow}
 Let $\Psi$ be  a lower frame sequence. If 
 $\overline{D_{\Psi}^r}$ has closed range, then the following are equivalent:
 \begin{enumerate}
     \item $\Psi$ is a complete Riesz-Fischer sequence.
     \item $\Psi$ is minimal.
     \item $\Psi$ is $\omega$-independent.
     \item $\Psi$ has a biorthogonal sequence.
     \item $\Psi$ has a unique biorthogonal sequence.
      \item $\Psi$ and $\wiSi$ are biorthogonal.
      \item $\Psi$ is an exact  sequence.
      \item $\Psi$ is a weak Schauder basis.
 \end{enumerate}
 \end{corollary}
 \begin{proof}
 $(1)\Leftrightarrow (2)$ is corollary \ref{RC=LM}.
 
 $(2)\Rightarrow (3)$ Clear.
 
 $(3)\Rightarrow (2)$ is \cite[Theorem 3.2]{casoleli1}.
 
 $(2)\Leftrightarrow (4)$, $(2)\Leftrightarrow (5)$ and $(2)\Leftrightarrow (6)$ follow by \cite[Lemma 3.3.1]{ole1n}.
 
   $(6)\Leftrightarrow (1)$ is clear by \cite[Lemma 3.3.1]{ole1n} and Corollary \ref{RC=LM}. 
 
 $(7)\Rightarrow (2)$ This is Corollary \ref{ex=min}.
 
  $(2)\Rightarrow (7)$ Assume that $\Psi$ is a minimal sequence in $\Hil$. Then for all $j\in I$, $\{\psi_i\}_{i\neq j}$  is not complete in $\Hil$ and so it is not a lower frame sequence, anymore.
  
   $(6)\Rightarrow (8)$ Assume that $\Psi$ and $\wiSi$ are biorthogonal. For all $f\in \Hil$ we have 
   $$f=\sum_{i\in I}\left<f,\wisi_i\right>\psi_i$$
   in the weak sense. In order to show that $\Psi$ is a basis in the weak sense, it is enough to show that if $f=\sum_{j\in I}c_j\psi_j$ in the weak sense, then these coefficients are  unique. For that
   \begin{eqnarray*}
   \left< f,\wisi_i\right>=\left<\sum_{j\in I} c_j\psi_j,\wisi_i\right>=\sum_{i\in I}c_j\left< \psi_j,\wisi_i\right>=c_i.
   \end{eqnarray*}

   $(8)\Rightarrow (6)$ Let $\Psi$ is a basis in the weak sense. Since $\psi_j=\sum_{i\in I}\left<\psi_j,\wisi_i\right>\psi_i=\sum_{i\in I}\delta_{ij}\psi_i$  in the weak sense and coefficients are unique, so $\left<\psi_j,\wisi_i\right>=\delta_{ij}$  and so $\Psi$ and $\widetilde{\Psi}$ are biorthogonal.
 \end{proof}

\section*{Acknowledgements}
We want to thank many colleagues for a lot of discussion on this topic in too many years: Yuri Latushkin, Pete Casazza, Diana Stoeva, Michael Speckbacher, Rosario Corso and Lukas K\"ohldorfer.
We also thank Rosario Corso for a first reading of the paper and comments.

The work on this manuscript was supported by the Innovation project of the Austrian Academy of Sciences (\verb:IF_2019_24_Fun:, "Frames and Unbounded Operators"). 

\bibliographystyle{abbrv}
\bibliography{biblioall}

\def\xxl{\color{red}}
\def\mitra{\color{olive}}
\def\xxb{\color{blue}}
\definecolor{darkcyan}{rgb}{0.0, 0.55, 0.55}
\def\xxg{\color{darkcyan}}
\definecolor{darkviolet}{rgb}{0.58,0,0.83} 
\def\xxlnew{\color{darkviolet}}

\end{document}